\documentclass[preprint,10pt]{elsarticle}

\usepackage{amssymb}
\usepackage{amsthm}
\usepackage{graphics}
\usepackage{amsmath}
\usepackage[dvipsnames,usenames]{color}

\newtheorem{thm}{Theorem}[section]
\newtheorem{cor}[thm]{Corollary}
\newtheorem{lem}[thm]{Lemma}

\newtheorem{ass}[thm]{Assumption}
\newtheorem{rem}{Remark}

\newtheorem{exa}[thm]{Example}
\theoremstyle{definition}

\numberwithin{equation}{section}

\journal{}

\begin{document}

\begin{frontmatter}



\title{Convergence rate in $\mathcal{L}^p$ sense of tamed EM scheme for highly nonlinear  neutral  multiple-delay stochastic McKean-Vlasov equations}

\author[label1]{Shuaibin Gao}
\author[label1]{Qian Guo}
\author[label2]{Junhao Hu}
\author[label3]{Chenggui Yuan}
%
\address[label1]{Department of Mathematics, Shanghai Normal University, Shanghai 200234, China}
\address[label2]{School of Mathematics and Statistics, South-Central University For Nationalities,
Wuhan 430074, China}
\address[label3]{Department of Mathematics, Swansea University, Bay Campus, Swansea, SA1 8EN, U. K.}

\begin{abstract}
This paper focuses on the numerical scheme of highly nonlinear neutral multiple-delay stohchastic McKean-Vlasov equation (NMSMVE) by virtue of the stochastic particle method.
First, under general assumptions, the results about propagation of chaos in $\mathcal{L}^p$ sense are shown.
Then the tamed Euler-Maruyama scheme to the
corresponding particle system is established and the convergence rate in $\mathcal{L}^p$ sense is obtained.
Furthermore, combining these two results gives the convergence error between the objective NMSMVE and  numerical approximation, which is related to the particle number and step size. Finally, two numerical examples are provided to support the finding.
\end{abstract}


\end{frontmatter}



\section{Introduction}

The theories of stochastic McKean-Vlasov equations (SMVEs) have been investigated by plenty of scholars, since SMVEs appear in many research fields, such as biological systems, chemistry and mean-field games \cite{2,3,1}. 
The salient feature of SMVEs is that the coefficients depend on the distributions of state variables, which brings difficulties to the research. 
SMVEs are also called distribution dependent stochastic differential equations (SDEs) or mean-field SDEs.
Reviewing the pioneering works, SMVEs were studied by McKean in \cite{4,5,6}, which were inspired by \cite{7}.
The existence and uniqueness of SMVEs were discussed in \cite{8,9,10}. 
As for other theories of SMVEs, we refer the readers to \cite{11,12,13,14,15,16}. 

As so often is the case, the true solutions to SMVEs cannot be expressed explicitly.
Hence, analyzing the numerical solutions is a common way to get the properties of the true solutions.
However, the classical Euler-Maruyama (EM) scheme cannot be used to simulate SDEs with superlinear coefficients well \cite{37}.
By borrowing the ideas in \cite{38,39,40}, the tamed EM scheme for SMVEs with superlinear drift coefficients was proposed in \cite{17}.
The tamed Milstein scheme for SMVEs was established to improve the convergence rate in \cite{18}.

When the time-delay is taken into consideration, the numerical schemes for SMVEs with delay were discussed in \cite{19,30,31,20}.
Neutral SMVEs with delay refer to a class of SMVEs which not only depend on the present, past state variables, but also contain derivatives with delay. 
Neutral SMVEs with delay was approximated by the tamed EM scheme \cite{21}.
However, the equation form is limited. For example, the following scalar equation is not included in \cite{21}:
	\begin{equation}\label{875678}
	\begin{split}
		d[Y(t)+Y^3(t-\rho)]
		&=[-2Y(t)+Y(t-\rho)-2Y^5(t-\rho)+\mathbb{E}|Y(t)|]dt\\
		&~~~+[Y(t)+Y(t-\rho)]dB(t).
	\end{split}
\end{equation}
Here, $\rho$ is the constant delay. Actually, some numerical schemes for neutral stochastic differential delay equations (NSDDEs), whose coefficients are not dependent of the distributions,  also have this limitation, such as \cite{36,42,43,44,41}.
To overcome this shortcoming and relax the constraint of the delay variables, we use the techniques in \cite{22,23} to approximate the highly nonlinear neutral SMVEs with delay.

In this paper, we focus on a class of highly nonlinear neutral multiple-delay stohchastic McKean-Vlasov equations (NMSMVEs) of the form:
\begin{equation}\label{sdde1}
	\begin{split}
		&d [Y(t)-D(Y(t-\rho))]\\
		&=\alpha\left( Y(t), Y(t-\rho_2),\cdots,Y(t-\rho_r), \mathbb{L}_{Y(t)}, \mathbb{L}_{Y(t-\rho_2)},\cdots,\mathbb{L}_{Y(t-\rho_r)}\right) d t\\
		&+\beta\left( Y(t), Y(t-\rho_2),\cdots,Y(t-\rho_r), \mathbb{L}_{Y(t)}, \mathbb{L}_{Y(t-\rho_2)},\cdots,\mathbb{L}_{Y(t-\rho_r)}\right) dB(t),
	\end{split}
\end{equation}
on $t\in[0,T]$, where  $\mathbb{L}_{X(t-\rho_v)}$ is the law of $X$ at time $t-\rho_v$ for $v\in \mathbb{S}_r:=\{1,2,\ldots,r\}$. 
Moreover, for $v\in\mathbb{S}_{r}$, let $0\leq\rho_v\leq\rho$.
Here, 
 $D:\mathbb{R}^{d}\rightarrow\mathbb{R}^{d},$ $\alpha:(\mathbb{R}^{d})^{r}\times(\mathcal{P}_{2}(\mathbb{R}^{d})
	)^{r}\rightarrow\mathbb{R}^{d},$
$\beta:(\mathbb{R}^{d})^{r}\times(\mathcal{P}_{2}(\mathbb{R}^{d})
)^{r}\rightarrow\mathbb{R}^{d\times m}.$

In reality, the multiple-delay systems are very significant, and they turn up on many occasions \cite{24,25,26,27}.
Then in this paper, the tamed EM scheme is established for NMSMVEs (\ref{sdde1}). 
Additionally, the convergence rate in $\mathcal{L}_p$ sense is shown, which is differential from the previous papers \cite{28,18,19,30,21,46,17}, where only the convergence rate in $\mathcal{L}_2$ sense was given. 
After analyzing the existing results, we find that the key to overcome this problem is to obtain the propagation of chaos in $\mathcal{L}_p$ sense. By virtue of the theory in \cite{29}, we give this result in Theorems \ref{thm4.1}  and \ref{thm4.2}.

All in all, the main contributions of the present paper can be stated as follows.
\begin{itemize}
	\item[$\bullet$] There is only one delay in \cite{30,21}, but we shall deal with multiple delays in (\ref{sdde1}). Moreover, the coefficients of (\ref{sdde1}) depend on the distributions of the delay variables.

	\item[$\bullet$] By borrowing the ideas in \cite{22,23}, the form of the equations and the constraint of the delay variables are more general, which are allowed highly nonlinear. 
	
	\item[$\bullet$] The requirement for the neutral term is also allowed highly nonlinear.
	
	\item[$\bullet$] The propagation of chaos in $\mathcal{L}_p$ sense is shown with the aid of the theory in \cite{29}.  The convergence rate in $\mathcal{L}_p$ sense of the tamed EM scheme is presented.
\end{itemize}

The organization for the rest of the paper is as follows. We simplfy the NMSMVEs by the projection operator and give the boundedness of the true solution in Section 2. The propagation of chaos in $\mathcal{L}_p$ sense is shown in Section 3. In Section 4, the tamed EM scheme is established to approximate  NMSMVEs. Section 5 contains an one-dimensional example and a two-dimensional example.

\section{Preliminaries}\label{section2}

Let $\big ( \Omega, \mathcal{F}, \{\mathcal{F}_t\}_{t\ge 0}, \mathbb{P} \big )$ be a complete probability space with a filtration  $\{\mathcal{F}_t\}_{t\ge 0}$ satisfying the usual conditions (i.e., it is increasing and right continuous while $ \mathcal{F}_0$ contains all $\mathbb{P}$-null sets).
For $x\in \mathbb{R}^d$, let $|x|$ be its Euclidean norm.
For the real numbers $b_1, b_2$, denote $b_1\wedge b_2 = \text {min}\{b_1, b_2\}$ and $b_1\vee b_2=\text{max}\{b_1,b_2\}$.
Let $\lfloor b_1  \rfloor$ be the largest integer that does not exceed $b_1$.
For a set $S$, define $\mathbb{I}_S(x)=1$ if $x \in S$ and $\mathbb{I}_S(x)=0$ if $x \notin S$ (i.e., $\mathbb{I}_S$ is indicator function).
Assume that $\mathcal{C}:=\mathcal{C}([-\rho,0];\mathbb{R}^d)$ is the family of all continuous functions $\varphi$ from $[-\rho,0] $ to $\mathbb{R}^d$ with the norm $\|\varphi\|=\text{sup}_{-\rho \le \theta \le 0}|\varphi (\theta)|$.
The probability expectation with respect to $\mathbb{P}$ is defined by $\mathbb{E}$.
For $p\geq1$, $\mathcal{L}^p :=\mathcal{L}^p( \Omega, \mathcal{F},  \mathbb{P} \big )$ is the set of random variables $X$ with $\mathbb{E}|X|^p<\infty$.
Let $B(t) $ be an $m$-dimensional Brownian motion on the probability space.
Denote $\mathbb{R}_+ = [0,+\infty)$.

Let $\delta_{x}(\cdot)$ stand for the Dirac measure at point $x\in \mathbb{R}^{d}$.
Assume that $\mathcal{P}(\mathbb{R}^{d})$ is the family of all probability measures on $\mathbb{R}^{d}$. For $q\geq1$, define
$$\mathcal{P}_{q}(\mathbb{R}^{d})=\left\{\mu\in\mathcal{P}(\mathbb{R}^{d}):\left(\int_{\mathbb{R}^{d}}|x|^{q}\mu(dx)\right)^{1/q}<\infty\right\},$$
and set $\mathcal{W}_{q}(\mu)=\left(\int_{\mathbb{R}^{d}}|x|^{q}\mu(dx)\right)^{1/q}$ for any $\mu\in \mathcal{P}_{q}(\mathbb{R}^{d})$.
For $q\geq1$,  the Wassertein distance of $\mu,\nu\in \mathcal{P}_{q}(\mathbb{R}^{d})$ is defined by
$$\mathbb{W}_{q}(\mu,\nu)=\inf_{\pi\in \mathfrak{C}(\mu,\nu)}\left(\int_{\mathbb{R}^{d}\times\mathbb{R}^{d}}|x-y|^{q}\pi(dx,dy)\right)^{1/q},$$
where $\mathfrak{C}(\mu,\nu)$ is the family of all couplings for $\mu,\nu$, i.e., $\pi(\cdot,\mathbb{R}^{d})=\mu(\cdot)$ and $\pi(\mathbb{R}^{d},\cdot)=\nu(\cdot)$.

We quote Lemma 2.3 in \cite{8} as the following lemma.
\begin{lem}
For any $\mu\in \mathcal{P}_{2}(\mathbb{R}^{d})$, we have $\mathbb{W}_{2}(\mu,\delta_{0})=\mathcal{W}_{2}(\mu)$.
\end{lem}

Define the segment process $y_{t}=\{y(t+\theta):-\rho\leq\theta\leq0\}$ for any $t\geq0$. 
Then $y_{t}\in \mathcal{C}$.
In order to simplify the equation form, we introduce the following projection operator. 
Let $\Gamma_{\theta}(\varphi):\mathcal{C}\rightarrow\mathbb{R}^{d}$, $\Gamma_{\theta}(\varphi)=\varphi(\theta)$ for $\varphi\in\mathcal{C}$ and $\theta\in[-\rho,0]$.
In addition, we set
$$\Gamma(\varphi)=(\Gamma_{\bar{s}_{1}}(\varphi),\Gamma_{\bar{s}_{2}}(\varphi),\cdots,\Gamma_{\bar{s}_{r}}(\varphi)),$$
$$\mathbb{L}_{\Gamma(\varphi)}=(\mathbb{L}_{\Gamma_{\bar{s}_{1}}(\varphi)},\mathbb{L}_{\Gamma_{\bar{s}_{2}}(\varphi)},\cdots,\mathbb{L}_{\Gamma_{\bar{s}_{r}}(\varphi)}),$$
for any $\varphi\in\mathcal{C}$ and $\bar{s}_{1},\bar{s}_{2},\ldots,\bar{s}_{r}\in[-\rho,0]$.
Let $\bar{s}_{1}=0$, $\bar{s}_{r}=-\rho$ throughout the paper.
For another, it should be noted that we can arrange the $r$-delays in (\ref{sdde1}) into a non subtractive sequence $\{\rho_1,\rho_2,\cdots,\rho_r\}$. Then set $\rho_1=-\bar{s}_{1}=0$,  $\rho_r=-\bar{s}_{r}=\rho$ and $\rho_v=-\bar{s}_{v}$, $v\in \{2, 3, \cdots, r-1\}$. 
Based on these notations, NMSMVE (\ref{sdde1}) can be rewritten as
\begin{equation}\label{2.2xin}
	\begin{split}
 &d[Y(t)-D(Y(t-\rho))]\\
 &=\alpha\left(\Gamma(Y_t),\mathbb{L}_{\Gamma(Y_t)}\right)dt+\beta\left(\Gamma(Y_t),\mathbb{L}_{\Gamma(Y_t)}\right) dB(t),~~~ t\in[0,T].
 \end{split}
\end{equation}
Here, the initial data $\{Y(\theta):\theta\in [-\rho,0]\}=\xi \in\mathcal{C}$.
Moreover,
 $D:\mathbb{R}^{d}\rightarrow\mathbb{R}^{d},$
$\alpha:\left(\mathbb{R}^{d}\right)^{r}\times\left(\mathcal{P}_{2}(\mathbb{R}^{d})\right)^{r}\rightarrow\mathbb{R}^{d},$
$\beta:\left(\mathbb{R}^{d}\right)^{r}\times\left(\mathcal{P}_{2}(\mathbb{R}^{d})\right)^{r}\rightarrow\mathbb{R}^{d\times m}$
are all Borel-measurable. 

\begin{rem}
One can observe that
\begin{equation*}
	\begin{split}
		\Gamma(Y_t)=&(\Gamma_{\bar{s}_{1}}(Y_{t}),\Gamma_{\bar{s}_{2}}(Y_{t}),\cdots,\Gamma_{\bar{s}_{r}}(Y_{t}))
		\\=&(Y_{t}(\bar{s}_{1}),Y_{t}(\bar{s}_{2}),\cdots,Y_{t}(\bar{s}_{r}))\\=&(Y(t),Y(t+\bar{s}_{2}),\cdots,Y(t+\bar{s}_{r}))\\=&(Y(t),Y(t-\rho_2),\cdots,Y(t-\rho)).
	\end{split}
\end{equation*}
\end{rem}
The Theorem 1 in \cite{29} is cited as the following theorem.
\begin{thm}\label{thm2.1}
Let $\{X_{n}\}_{n\geq1}$ be a sequence of independent and identically distributed (i.i.d.) random variables in $\mathbb{R}^{d}$ with the distribution $\mu\in \mathcal{P}_{\bar{p}}(\mathbb{R}^{d})$ and define the empirical measure $\mu_{N}=\frac{1}{N}\sum_{j=1}^{N}\delta_{X_{j}}$ for $j\in \mathbb{S}_N$. Then for $\bar{p}>p\geq2$, there exists a constant $C_{p,\bar{p},d}$ depending on $p$, $\bar{p}$, $d$, such that, for all $N\geq1$,
\begin{equation*}
\begin{split}
\mathbb{E}\left(\mathbb{W}_{p}^{p}(\mu _{N},\mu )\right)\leq& C_{p,\bar{p},d}
\left\{\begin{array}{ll}
N^{-1 / 2}+N^{-(\bar{p}-p)/\bar{p}}, & \text { if } p>d/2\text{ and }\bar{p}\neq2p, \\
N^{-1 / 2} \log(1+N)+N^{-(\bar{p}-p)/\bar{p}}, & \text { if } p=d/2\text{ and }\bar{p}\neq2p,\\
N^{-p / d}+N^{-(\bar{p}-p)/\bar{p}}, & \text { if }2\leq p<d/2.
\end{array}\right.
\end{split}
\end{equation*}
\end{thm}

Then the special case is stated as the following corollary.
\begin{cor}\label{cor2.2}
Assume that the settings in Theorem \ref{thm2.1} holds. Then for $\bar{p}>2p$ and $p\geq2$, there exists a constant $C_{p,\bar{p},d}$  depending on $p$, $\bar{p}$, $d$, such that, for all $N\geq1$,
\begin{equation*}
\begin{split}
\mathbb{E}\left(\mathbb{W}_{p}^{p}(\mu _{N},\mu )\right)\leq& C_{p,\bar{p},d}
\left\{\begin{array}{ll}
N^{-1 / 2}, & \text { if } p>d/2, \\
N^{-1 / 2} \log(1+N), & \text { if } p=d/2,\\
N^{-p / d}, & \text { if }2\leq p<d/2.
\end{array}\right.
\end{split}
\end{equation*}
\end{cor}

In order to make assumptions on coefficients, we denote  $U_{i}: \mathbb{R}^{d}\times\mathbb{R}^{d}\rightarrow\mathbb{R}$, $i=1,2,3$ and  there exist constants $K_{U}>0$ and $l_{i}\geq 1$ such that
\begin{equation}\label{2.10}
	0\leq U_{i}(\bar{x},\hat{x})\leq K_{U}(1+|\bar{x}|^{l_{i}}+|\hat{x}|^{l_{i}}),
\end{equation}
for any $\bar{x},\hat{x}\in \mathbb{R}^{d}$ and $i=1,2,3$.

Let $x^{(r)}=\left(x_{1},x_{2},\cdots,x_{r}\right)$, $y^{(r)}=\left(y_{1},y_{2},\cdots,y_{r}\right)$ for $x_{i}$, $y_{i}\in \mathbb{R}^{d}$, $i\in \mathbb{S}_r$, and $\mu^{(r)}=\left(\mu_{1},\mu_{2},\cdots,\mu_{r}\right)$, $\nu^{(r)}=\left(\nu_{1},\nu_{2},\cdots,\nu_{r}\right)$ for $\mu_{i}$, $\nu_{i}\in\mathcal{P}_{2}(\mathbb{R}^{d})$, $i\in \mathbb{S}_r$.

\begin{ass}\label{Halpha}
There exists a constant $K_1>0$ such that
\begin{equation*}
 \left|\alpha(x^{(r)},\mu^{(r)})-\alpha(y^{(r)},\mu^{(r)})\right|\leq K_{1}\sum_{i=1}^{r}\left[U_{1}(x_{i},y_{i})|x_{i}-y_{i}|\right],
\end{equation*}
\begin{equation*}
 \left|\alpha(x^{(r)},\mu^{(r)})-\alpha(x^{(r)},\nu^{(r)})\right|\leq K_{1}\sum_{i=1}^{r}\mathbb{W}_{2}(\mu _{i},\nu_{i} ),
\end{equation*}
\begin{equation*}
\begin{split}
 &\left(x_{1}-D(x_{r})-y_{1}+D(y_{r})\right)^{T}\left(\alpha(x^{(r)},\mu^{(r)})-\alpha(y^{(r)},\nu^{(r)})\right)\\&\leq K_{1}\left[\sum_{i=1}^{r-1}|x_{i}-y_{i}|^{2}+U_{2}^{2}(x_{r},y_{r})|x_{r}-y_{r}|^{2}+\sum_{i=1}^{r}\mathbb{W}_{2}^{2}(\mu _{i},\nu_{i} )\right],
 \end{split}
\end{equation*}
for any $x^{(r)}, y^{(r)}\in (\mathbb{R}^{d})^{r}$ and $\mu^{(r)}, \nu^{(r)}\in (\mathcal{P}_{2}(\mathbb{R}^{d}))^{r}$.
\end{ass}

\begin{ass}\label{Hbeta}
There exists a constant $K_2>0$ such that
\begin{equation*}
 \left|\beta(x^{(r)},\mu^{(r)})-\beta(y^{(r)},\mu^{(r)})\right|\leq K_{2}\left[\sum_{i=1}^{r-1}|x_{i}-y_{i}|+U_{2}(x_{r},y_{r})|x_{r}-y_{r}|\right],
\end{equation*}
\begin{equation*}
 \left|\beta(x^{(r)},\mu^{(r)})-\beta(x^{(r)},\nu^{(r)})\right|\leq K_{2}\sum_{i=1}^{r}\mathbb{W}_{2}(\mu _{i},\nu_{i} ),
\end{equation*}
for any $x^{(r)}, y^{(r)}\in (\mathbb{R}^{d})^{r}$ and $\mu^{(r)}, \nu^{(r)}\in (\mathcal{P}_{2}(\mathbb{R}^{d}))^{r}$.
\end{ass}

\begin{ass}\label{HD}
$D(0)=0$ and there exists a constant $K_3>0$ such that
\begin{equation*}
 \left|D(x_{r})-D(y_{r})\right|\leq K_3 U_{3}(x_{r},y_{r})|x_{r}-y_{r}|,
\end{equation*}
for any  $x_r$, $y_r\in \mathbb{R}^{d}$.
\end{ass}

\begin{ass}\label{initial}
	There exists a constant $K_4>0$ such that
	\begin{equation*}
		\left|\xi(t_1)-\xi(t_2)\right|\leq K_{4}\left|t_1-t_2\right|^{\frac{1}{2}},
	\end{equation*}
	for all $t_1, t_2\in[-\rho,0]$.
\end{ass}

By Assumptions \ref{Halpha}-\ref{HD}, one can see there exist some constants $\bar{C}_{1}, \bar{C}_{2}, \bar{C}_{3}$ such that
\begin{equation}\label{tuichu1}
 \left|\alpha(x^{(r)},\mu^{(r)})\right|\leq \bar{C}_{1}\left(1+\sum_{i=1}^{r}\left[U_{1}(x_{i},0)|x_{i}|+\mathcal{W}_{2}(\mu _{i} )\right]\right),
\end{equation}
\begin{equation}\label{tuichu2}
\begin{split}
 &\left(x_{1}-D(x_{r})\right)^{T}\alpha(x^{(r)},\mu^{(r)})\leq \bar{C}_{2}\left[1+\sum_{i=1}^{r-1}|x_{i}|^{2}+U_{2}^{2}(x_{r},0)|x_{r}|^{2}+\sum_{i=1}^{r}\mathcal{W}_{2}^{2}(\mu _{i})\right],
 \end{split}
\end{equation}
\begin{equation}\label{tuichu3}
 \left|\beta(x^{(r)},\mu^{(r)})\right|^{2}\leq \bar{C}_{3}\left[1+\sum_{i=1}^{r-1}|x_{i}|^{2}+U_{2}^{2}(x_{r},0)|x_{r}|^{2}+\sum_{i=1}^{r}\mathcal{W}_{2}^{2}(\mu _{i})\right],
\end{equation}
\begin{equation}\label{tuichu4}
	\left|D(x_{r})\right|\leq K_3 U_{3}(x_{r},0)|x_{r}|,
\end{equation}
for any $x^{(r)}\in (\mathbb{R}^{d})^{r}$ and $\mu^{(r)}\in (\mathcal{P}_{2}(\mathbb{R}^{d}))^{r}$.
In the rest of this paper, set $l_{U}=max\{l_{1},l_{2},l_{3}\}$ for simplicity.

\begin{thm}\label{thm3.1}
Let Assumptions \ref{Halpha}-\ref{HD} hold. Then there exists a unique strong solution $Y(t)$ to \eqref{2.2xin} and $Y(t)$ satisfies, for any $\bar{p}> 0$ and $T>0$,
\begin{equation*}
 \mathbb{E}\left(\sup_{0\leq t\leq T}|Y(t)|^{\bar{p}}\right)\leq C.
\end{equation*}
\end{thm}
\begin{proof}
The NMSMVEs (\ref{2.2xin}) admits a unique strong solution by Theorem 3.5 in \cite{21} with the assistance of Lemma 2.1 in \cite{33}.
We just give the detailed proof of the boundedness of the solution.
Set $\bar{p}\geq 2$ first. For any $t\in[0,T]$, by It\^{o}'s formula, we have
\begin{equation*}
\begin{split}
&\left|Y(t)-D(Y(t-\rho))\right|^{\bar{p}}-\left|\xi(0)-D(\xi(-\rho))\right|^{\bar{p}}
\\&\leq\bar{p}\int_{0}^{t}\left|Y(s)-D(Y(s-\rho))\right|^{\bar{p}-2}\left(Y(s)-D(Y(s-\rho))\right)^{T}\alpha\left(\Gamma(Y_{s}),\mathbb{L}_{\Gamma(Y_{s})}\right)ds
\\&+ \frac{\bar{p}(\bar{p}-1)}{2}\int_{0}^{t}\left|Y(s)-D(Y(s-\rho))\right|^{\bar{p}-2}|\beta\left(\Gamma(Y_{s}),\mathbb{L}_{\Gamma(Y_{s})}\right)|^{2}ds
\\&+\bar{p}\int_{0}^{t}\left|Y(s)-D(Y(s-\rho))\right|^{\bar{p}-2}\left(Y(s)-D(Y(s-\rho))\right)^{T}\beta\left(\Gamma(Y_{s}),\mathbb{L}_{\Gamma(Y_{s})}\right)dB(s)
\\&=:I_{1}(t)+I_{2}(t)+I_{3}(t).
\end{split}
\end{equation*}
Define the stopping time $$\tau_{\tilde{N}}=T\wedge inf\{t\in [0,T]:|Y(t)|\geq\tilde{N}\},$$ for every integer $\tilde{N}\geq1$. Obviously, $\tau_{\tilde{N}}\uparrow T$ a.s.
By (\ref{tuichu2}), (\ref{tuichu3}), H\"{o}lder's inequality and Young's inequality, we derive that
\begin{equation*}
\begin{split}
\mathbb{E}&\left[\sup_{0\leq s\leq t\wedge \tau_{\tilde{N}}}(I_{1}(s)+I_{2}(s))\right]\\\leq& C\mathbb{E}\int_{0}^{t\wedge \tau_{\tilde{N}}}\left|Y(s)-D(Y(s-\rho))\right|^{\bar{p}-2}
\\&\left[1+\sum_{i=1}^{r-1}|Y(s+\bar{s}_{i})|^{2}
+U_{2}^{2}(Y(s-\rho),0)|Y(s-\rho)|^{2}+\sum_{i=1}^{r}\mathcal{W}_{2}^{2}(\mathbb{L}_{Y(s+\bar{s}_{i})})\right]ds
\\\leq& C\mathbb{E}\int_{0}^{t\wedge \tau_{\tilde{N}}}\left(|Y(s)|^{\bar{p}-2}+U^{\bar{p}-2}_{3}(Y(s-\rho),0)|Y(s-\rho)|^{\bar{p}-2}\right)
\\&\left[1+\sum_{i=1}^{r-1}|Y(s+\bar{s}_{i})|^{2}
+U_{2}^{2}(Y(s-\rho),0)|Y(s-\rho)|^{2}+\sum_{i=1}^{r}\mathcal{W}_{2}^{2}(\mathbb{L}_{Y(s+\bar{s}_{i})})\right]ds
\\\leq& C\mathbb{E}\int_{0}^{t\wedge \tau_{\tilde{N}}}\left[1+\sum_{i=1}^{r-1}|Y(s+\bar{s}_{i})|^{\bar{p}}+\|\xi\|^{\bar{p}}\right.
\\&\left.+\left(U_{2}(Y(s-\rho),0)\vee U_{3}(Y(s-\rho),0)\right)^{\bar{p}}|Y(s-\rho)|^{\bar{p}}+\sum_{i=1}^{r}\mathcal{W}_{2}^{p}(\mathbb{L}_{Y(s+\bar{s}_{i})})\right]ds
\\\leq& C\mathbb{E}\int_{0}^{t}\left(1+\sup_{0\leq u\leq s\wedge \tau_{\tilde{N}}}|Y(u)|^{\bar{p}}+|Y(s-\rho)|^{(l_{U}+1)\bar{p}}\right)ds.
\end{split}
\end{equation*}
Using Young's inequality, H\"{o}lder's inequality and BDG's inequality gives that
\begin{equation*}
\begin{split}
\mathbb{E}&\left[\sup_{0\leq s\leq t\wedge \tau_{\tilde{N}}}I_{3}(s)\right]\\\leq& C\mathbb{E}\left[\int_{0}^{t\wedge \tau_{\tilde{N}}}\left|Y(s)-D(Y(s-\rho))\right|^{2\bar{p}-2}
\left|\beta\left(\Gamma(Y_{s}),\mathbb{L}_{\Gamma(Y_{s})}\right)\right|^{2}ds\right]^{1/2}
\\\leq& \frac{1}{2}\mathbb{E}\left(\sup_{0\leq s\leq t\wedge \tau_{\tilde{N}}}\left|Y(s)-D(Y(s-\rho))\right|^{\bar{p}}\right)
\\&+C\mathbb{E}\left(\int_{0}^{t\wedge \tau_{\tilde{N}}}\left[1+\sum_{i=1}^{r-1}|Y(s+\bar{s}_{i})|^{2}
+U_{2}^{2}(Y(s-\rho),0)|Y(s-\rho)|^{2}\right.\right.
\\&\left.\left.+\sum_{i=1}^{r}\mathcal{W}_{2}^{2}(\mathbb{L}_{Y(s+\bar{s}_{i})})\right]ds\right)^{\bar{p}/2}
\\\leq& \frac{1}{2}\mathbb{E}\left(\sup_{0\leq s\leq t\wedge \tau_{\tilde{N}}}\left|Y(s)-D(Y(s-\rho))\right|^{\bar{p}}\right)\\
&+C\mathbb{E}\int_{0}^{t}\left(1+\sup_{0\leq u\leq s\wedge \tau_{\tilde{N}}}|Y(u)|^{\bar{p}}+|Y(s-\rho)|^{(l_{U}+1)\bar{p}}\right)ds.
\end{split}
\end{equation*}
Thus,
\begin{equation*}
\begin{split}
\mathbb{E}&\left(\sup_{0\leq u\leq t\wedge \tau_{\tilde{N}}}\left|Y(u)\right|^{\bar{p}}\right)
\\\leq&C\left[\mathbb{E}\left(\sup_{0\leq u\leq t\wedge \tau_{\tilde{N}}}\left|D(Y(u-\rho))\right|^{\bar{p}}\right)
+\mathbb{E}\left(\sup_{0\leq u\leq t\wedge \tau_{\tilde{N}}}\left|Y(u)-D(Y(u-\rho))\right|^{\bar{p}}\right)\right]
\\\leq&C\left[1+\mathbb{E}\int_{0}^{t}\left(\sup_{0\leq u\leq s\wedge \tau_{\tilde{N}}}\left|Y(u)\right|^{\bar{p}}\right)ds
+\mathbb{E}\int_{0}^{t\wedge \tau_{\tilde{N}}}|Y(s-\rho)|^{(l_{U}+1)\bar{p}}ds\right.
\\&\left.+\mathbb{E}\left(\sup_{0\leq u\leq t\wedge \tau_{\tilde{N}}}\left|Y(u-\rho)\right|^{(l_{U}+1)\bar{p}}\right)\right].
\end{split}
\end{equation*}
Thanks to Gronwall's inequality, we get that
\begin{equation}\label{*1}
\mathbb{E}\left(\sup_{0\leq u\leq t\wedge \tau_{\tilde{N}}}\left|Y(u)\right|^{\bar{p}}\right)
\leq C+C\mathbb{E}\left(\sup_{0\leq u\leq t\wedge \tau_{\tilde{N}}}\left|Y(u-\rho)\right|^{\bar{p}l_{*}}\right),
\end{equation}
where $l_{*}=l_{U}+1$.
Define a sequence $\bar{p}_{j}$ by 
\begin{equation*}
	\bar{p}_{j}=(2-j+\lfloor\frac{T}{\rho}\rfloor)\bar{p}l_{*}^{1-j+\lfloor\frac{T}{\rho}\rfloor},
\quad j=1,2,\ldots,\lfloor\frac{T}{\rho}\rfloor+1.
\end{equation*}
One can see that $\bar{p}_{j+1}l_{*}<\bar{p}_{j}$ for $j=1,2,\ldots,\lfloor\frac{T}{\rho}\rfloor+1$ and $\bar{p}_{\lfloor\frac{T}{\rho}\rfloor+1}=\bar{p}$.
For $u\in[0,\rho]$, \eqref{*1} means that \begin{equation*}
	\mathbb{E}\left(\sup_{0\leq u\leq \rho\wedge \tau_{\tilde{N}}}\left|Y(u)\right|^{\bar{p}_{1}}\right)
\leq C.
\end{equation*}
Then for $u\in [0,2\rho]$, using $\bar{p}_{2}l_{*}<\bar{p}_{1}$ and \eqref{*1} with H\"older's inequality gives that
\begin{equation*}
	\begin{split}
\mathbb{E}\left(\sup_{0\leq u\leq 2\rho\wedge \tau_{\tilde{N}}}\left|Y(u)\right|^{\bar{p}_{2}}\right)
&\leq C+C\mathbb{E}\left(\sup_{0\leq u\leq 2\rho\wedge \tau_{\tilde{N}}}\left|Y(u-\rho)\right|^{\bar{p}_{2}l_*}\right)\\
&\leq C+C\left[\mathbb{E}\left(\sup_{0\leq u\leq 2\rho\wedge \tau_{\tilde{N}}}\left|Y(u-\rho)\right|^{\bar{p}_{1}}\right)\right]^{\frac{\bar{p}_{2}l_{*}}{\bar{p}_{1}}}
\leq C.
\end{split}
\end{equation*}
By induction, we get that
\begin{equation*}
\mathbb{E}\left(\sup_{0\leq u\leq \left[\left((\lfloor\frac{T}{\rho}\rfloor+1)\rho\right)\wedge \tau_{\tilde{N}}\right]}\left|Y(u)\right|^{\bar{p}}\right)
\leq C.
\end{equation*}
The Fatou lemma leads to
\begin{equation*}
	\mathbb{E}\left(\sup_{0\leq u\leq \left[(\lfloor\frac{T}{\rho}\rfloor+1)\rho\right]}\left|Y(u)\right|^{\bar{p}}\right)
\leq C.
\end{equation*}
When $\bar{p}\in(0,2)$, the desired result follows by the H\"{o}lder inequality.
\end{proof}

\section{Propagation of Chaos}\label{Propagation of Chaos}

In this section, we will use the stochastic particle method in \cite{35,34} to approximate NMSMVEs \eqref{2.2xin}.
For any $i\in \mathbb{S}_N$, let $(B^{i},\xi^{i})$ be independent copies of $(B,\xi)$ and all $(B^{i},\xi^{i})$ are i.i.d. Moreover, for $\xi\in \mathcal{C}$, set $\left|\xi^{i}(t_{1})-\xi^{i}(t_{2})\right|\leq |t_{1}- t_{2}|^{1/2}, \forall t_{1}, t_{2}\in[-\rho,0]$.
A non-interacting particle system is given by
\begin{equation}\label{4.1}
d[Y^{i}(t)-D(Y^{i}(t-\rho))]=\alpha\left(\Gamma(Y^{i}_{t}),\mathbb{L}_{\Gamma(Y^{i}_{t})}\right)dt
+\beta\left(\Gamma(Y^{i}_{t}),\mathbb{L}_{\Gamma(Y^{i}_{t})}\right) dB^{i}_{t},
\end{equation}
with the intial value $\xi^{i}$, where
\begin{equation*}
	\begin{split}
	\Gamma(Y^{i}_{t})=(\Gamma_{\bar{s}_{1}}(Y^{i}_{t}),\cdots,\Gamma_{\bar{s}_{r}}(Y^{i}_{t}))
=(Y^{i}(t),\cdots,Y^{i}(t-\rho)),\\
\mathbb{L}_{\Gamma(Y^{i}_{t})}=(\mathbb{L}_{\Gamma_{\bar{s}_{1}}(Y^{i}_{t})},\cdots,\mathbb{L}_{\Gamma_{\bar{s}_{r}}(Y^{i}_{t})})
=(\mathbb{L}_{Y^{i}(t)},\cdots,\mathbb{L}_{Y^{i}(t-\rho)}).
\end{split}
\end{equation*}
One can see that $\mathbb{L}_{\Gamma(Y^{i}_{t})}=\mathbb{L}_{\Gamma(Y_{t})}$, $i\in \mathbb{S}_N$.
To deal with $\mathbb{L}_{\Gamma(Y_{t}^i)}$, we introduce the following interacting particle system of the form
\begin{equation}\label{4.2}
d[Y^{i,N}(t)-D(Y^{i,N}(t-\rho))]=\alpha\left(\Gamma(Y^{i,N}_{t}),\mathbb{L}_{\Gamma(Y^{N}_{t})}\right)dt
+\beta\left(\Gamma(Y^{i,N}_{t}),\mathbb{L}_{\Gamma(Y^{N}_{t})}\right) dB^{i}(t),
\end{equation}
with the intial value $\xi^{i}$, where
$$\Gamma(Y^{i,N}_{t})=(\Gamma_{\bar{s}_{1}}(Y^{i,N}_{t}),\cdots,\Gamma_{\bar{s}_{r}}(Y^{i,N}_{t}))
=(Y^{i,N}(t),\cdots,Y^{i,N}(t-\rho)),$$
$$\mathbb{L}_{\Gamma(Y^{N}_{t})}=(\mathbb{L}_{\Gamma_{\bar{s}_{1}}(Y^{N}_{t})},\cdots,\mathbb{L}_{\Gamma_{\bar{s}_{r}}(Y^{N}_{t})})
=(\mathbb{L}_{Y^{N}(t)},\cdots,\mathbb{L}_{Y^{N}(t-\rho)}),$$
and
 $$\mathbb{L}_{Y^{N}(t-\bar{s}_{v})}(\cdot):=\frac{1}{N}\sum_{j=1}^{N}\delta_{Y^{j,N}(t-\bar{s}_{v})}(\cdot), \quad v\in \mathbb{S}_r.$$
In the following of this paper, let $p\geq2$. The theory of the propagation of chaos is stated as the following theorem.

\begin{thm}\label{thm4.1}
Let Assumptions \ref{Halpha}-\ref{HD} hold and $(pl_{U}+\varepsilon)p<\varepsilon\bar{p}$ hold for $\varepsilon\in(0,1]$. Then there exists a constant $C$ indepentent of $N$ such that, for any $i\in \mathbb{S}_N$,
\begin{equation*}
\begin{split}
\mathbb{E}\left(\sup_{0\leq t\leq T}|Y^{i}(t)-Y^{i,N}(t)|^{p}\right)\leq C
\left\{\begin{array}{ll}
(N^{-1 / 2})^{\lambda_{T,\rho,p}}, & \text { if } p>d/2, \\
{[N^{-1 / 2} \log(1+N)]}^{\lambda_{T,\rho,p}}, & \text { if } p=d/2,\\
{(N^{-p / d})}^{\lambda_{T,\rho,p}}, & \text { if }2\leq p<d/2,
\end{array}\right.
\end{split}
\end{equation*}
where $\lambda_{T,\rho,p}=(\frac{p-\varepsilon}{p})^{\lfloor\frac{T}{\rho}\rfloor}$.
\end{thm}
\begin{proof}
For any $i\in \mathbb{S}_{N}$ and $t\in[0,T]$, set $$\Xi^{i}(t)=Y^{i}(t)-D(Y^{i}(t-\rho))-Y^{i,N}(t)+D(Y^{i,N}(t-\rho)).$$
Then using It\^{o}'s formula leads to
\begin{equation*}
\begin{split}
&|\Xi^{i}(t)|^{p}-|\Xi^{i}(0)|^{p}\\&\leq p\int_{0}^{t}|\Xi^{i}(s)|^{p-2}(\Xi^{i}(s))^{T}\left[\alpha\left(\Gamma(Y^{i}_{s}),\mathbb{L}_{\Gamma(Y^{i}_{s})}\right)
-\alpha\left(\Gamma(Y^{i,N}_{s}),\mathbb{L}_{\Gamma(Y^{N}_{s})}\right)\right]ds
\\&+\frac{p(p-1)}{2}\int_{0}^{t}|\Xi^{i}(s)|^{p-2}\left|\beta\left(\Gamma(Y^{i}_{s}),\mathbb{L}_{\Gamma(Y^{i}_{s})}\right)
-\beta\left(\Gamma(Y^{i,N}_{s}),\mathbb{L}_{\Gamma(Y^{N}_{s})}\right)\right|^{2}ds
\\&+p\int_{0}^{t}|\Xi^{i}(s)|^{p-2}(\Xi^{i}(s))^{T}
\\&\quad \quad \quad\left[\beta\left(\Gamma(Y^{i}_{s}),\mathbb{L}_{\Gamma(Y^{i}_{s})}\right)
-\beta\left(\Gamma(Y^{i,N}_{s}),\mathbb{L}_{\Gamma(Y^{N}_{s})}\right)\right]dB^{i}(s)
\\&=:J_{1}^{i}(t)+J_{2}^{i}(t)+J_{3}^{i}(t).
\end{split}
\end{equation*}
For $\varepsilon\in(0,1]$, we get from H\"{o}lder's inequality, Young's inequality and  Assumptions \ref{Halpha}, \ref{Hbeta} that
\begin{equation*}
\begin{split}
\mathbb{E}&\left[\sup_{0\leq s\leq t}(J_{1}^{i}(s)+J_{2}^{i}(s))\right]\\\leq& C\mathbb{E}\int_{0}^{t}\left|\Xi^{i}(s)\right|^{p-2}
\left[\sum_{v=1}^{r-1}|Y^{i}(s+\bar{s}_{v})-Y^{i,N}(s+\bar{s}_{v})|^{2}\right.
\\&+U_{2}^{2}(Y^{i}(s-\rho),Y^{i,N}(s-\rho))|Y^{i}(s-\rho)-Y^{i,N}(s-\rho)|^{2}
\\&\left.+\sum_{v=1}^{r}\mathbb{W}_{2}^{2}(\mathbb{L}_{Y^{i}(s+\bar{s}_{v})},\mathbb{L}_{Y^{N}(s+\bar{s}_{v})})\right]ds
\\\leq& C\mathbb{E}\int_{0}^{t}|\Xi^{i}(s)|^{p}ds
+C\mathbb{E}\int_{0}^{t}\left[\sum_{v=1}^{r-1}|Y^{i}(s+\bar{s}_{v})-Y^{i,N}(s+\bar{s}_{v})|^{p}\right.
\\&+U^{p}_{2}(Y^{i}(s-\rho),Y^{i,N}(s-\rho))|Y^{i}(s-\rho)-Y^{i,N}(s-\rho)|^{p}
\\&\left.+\sum_{v=1}^{r}\mathbb{W}_{p}^{p}(\mathbb{L}_{Y^{i}(s+\bar{s}_{v})},\mathbb{L}_{Y^{N}(s+\bar{s}_{v})})\right]ds
\\\leq& C\mathbb{E}\int_{0}^{t}\left|\Xi^{i}(s)\right|^{p}ds
+C\int_{0}^{t}\mathbb{E}\left(\sup_{0\leq u\leq s}|Y^{i}(u)-Y^{i,N}(u)|^{p}\right)ds
\\&+C\int_{0}^{t}\left[\mathbb{E}\left(1+|Y^{i}(s-\rho)|^{l_{U}p+\varepsilon}
+|Y^{i,N}(s-\rho)|^{l_{U}p+\varepsilon}\right)^{p/\varepsilon}\right]^{\varepsilon/p}
\\&\quad \quad \quad\cdot\left[\mathbb{E}|Y^{i}(s-\rho)-Y^{i,N}(s-\rho)|^{p}\right]^{(p-\varepsilon)/p}ds
\\&+C\mathbb{E}\int_{0}^{t}\sum_{v=1}^{r}\mathbb{W}_{p}^{p}(\mathbb{L}_{Y^{i}(s+\bar{s}_{v})},\mathbb{L}_{Y^{N}(s+\bar{s}_{v})})ds
\\
\end{split}
\end{equation*}
\begin{equation*}
	\begin{split}
\leq& C\mathbb{E}\int_{0}^{t}\left|\Xi^{i}(s)\right|^{p}ds
+C\int_{0}^{t}\mathbb{E}\left(\sup_{0\leq u\leq s}|Y^{i}(u)-Y^{i,N}(u)|^{p}\right)ds
\\&+C\int_{0}^{t}\left[\mathbb{E}|Y^{i}(s-\rho)-Y^{i,N}(s-\rho)|^{p}\right]^{(p-\varepsilon)/p}ds
\\&+C\mathbb{E}\int_{0}^{t}\sum_{v=1}^{r}\mathbb{W}_{p}^{p}(\mathbb{L}_{Y^{i}(s+\bar{s}_{v})},\mathbb{L}_{Y^{N}(s+\bar{s}_{v})})ds.
\end{split}
\end{equation*}
By Assumption \ref{Hbeta}, BDG's inequality, Young's inequality and  H\"{o}lder's inequality, we derive that
\begin{equation*}
\begin{split}
\mathbb{E}&\left[\sup_{0\leq s\leq t}J_{3}^{i}(s)\right]\\\leq&C\mathbb{E}\left[\int_{0}^{t}|\Xi^{i}(s)|^{2p-2}\left|\beta\left(\Gamma(Y^{i}_{s}),\mathbb{L}_{\Gamma(Y^{i}_{s})}\right)
-\beta\left(\Gamma(Y^{i,N}_{s}),\mathbb{L}_{\Gamma(Y^{N}_{s})}\right)\right|^{2}ds\right]^{1/2}
\\\leq&\frac{1}{2}\mathbb{E}\left(\sup_{0\leq s\leq t}|\Xi^{i}(s)|^{p}\right)
+C\mathbb{E}\left[\int_{0}^{t}\left[\sum_{v=1}^{r-1}|Y^{i}(s+\bar{s}_{v})-Y^{i,N}(s+\bar{s}_{v})|^{2}\right.\right.
\\&~~~+U_{2}^{2}(Y^{i}(s-\rho),Y^{i,N}(s-\rho))|Y^{i}(s-\rho)-Y^{i,N}(s-\rho)|^{2}
\\&~~~\left.\left.+\sum_{v=1}^{r}\mathbb{W}_{2}^{2}(\mathbb{L}_{Y^{i}(s+\bar{s}_{v})},\mathbb{L}_{Y^{N}(s+\bar{s}_{v})})\right] ds\right]^{p/2}
\\\leq&\frac{1}{2}\mathbb{E}\left(\sup_{0\leq s\leq t}|\Xi^{i}(s)|^{p}\right)+C\int_{0}^{t}\mathbb{E}\left(\sup_{0\leq u\leq s}|Y^{i}(u)-Y^{i,N}(u)|^{p}\right)ds
\\&+C\int_{0}^{t}\left[\mathbb{E}|Y^{i}(s-\rho)-Y^{i,N}(s-\rho)|^{p}\right]^{(p-\varepsilon)/p}ds
\\&+C\mathbb{E}\int_{0}^{t}\sum_{v=1}^{r}\mathbb{W}_{p}^{p}(\mathbb{L}_{Y^{i}(s+\bar{s}_{v})},\mathbb{L}_{Y^{N}(s+\bar{s}_{v})})ds.
\end{split}
\end{equation*}
Thanks to Gronwall's inequality, we have
\begin{equation*}
\begin{split}
\mathbb{E}\left(\sup_{0\leq s\leq t}|\Xi^{i}(s)|^{p}\right)\leq&C\int_{0}^{t}\mathbb{E}\left(\sup_{0\leq u\leq s}|Y^{i}(u)-Y^{i,N}(u)|^{p}\right)ds
\\&+C\int_{0}^{t}\left[\mathbb{E}|Y^{i}(s-\rho)-Y^{i,N}(s-\rho)|^{p}\right]^{(p-\varepsilon)/p}ds
\\&+C\mathbb{E}\int_{0}^{t}\sum_{v=1}^{r}\mathbb{W}_{p}^{p}(\mathbb{L}_{Y^{i}(s+\bar{s}_{v})},\mathbb{L}_{Y^{N}(s+\bar{s}_{v})})ds.
\end{split}
\end{equation*}
Therefore, we get from Assumption \ref{HD} and the technique in the estimation of $J_{1}^{i}(t)+J_{2}^{i}(t)$ that
\begin{equation*}
\begin{split}
\mathbb{E}&\left(\sup_{0\leq s\leq t}|Y^{i}(s)-Y^{i,N}(s)|^{p}\right)
\\\leq&C\mathbb{E}\left(\sup_{0\leq s\leq t}|\Xi^{i}(s)|^{p}\right)+C\mathbb{E}\left(\sup_{0\leq s\leq t}|D(Y^{i}(s-\rho))- D(Y^{i,N}(s-\rho))|^{p}\right)
\\\leq&C\int_{0}^{t}\mathbb{E}\left(\sup_{0\leq u\leq s}|Y^{i}(u)-Y^{i,N}(u)|^{p}\right)ds
\\&+C\left[\mathbb{E}\left(\sup_{0\leq u\leq t}|Y^{i}(u-\rho)-Y^{i,N}(u-\rho)|^{p}\right)\right]^{(p-\varepsilon)/p}
\\&+C\mathbb{E}\int_{0}^{t}\sum_{v=1}^{r}\mathbb{W}_{p}^{p}(\mathbb{L}_{Y^{i}(s+\bar{s}_{v})},\mathbb{L}_{Y^{N}(s+\bar{s}_{v})})ds.
\end{split}
\end{equation*}
Using Gronwall's inequality again yields that
\begin{equation}\label{*2}
\begin{split}
\mathbb{E}&\left(\sup_{0\leq u\leq t}|Y^{i}(u)-Y^{i,N}(u)|^{p}\right)
\\\leq&C\left[\mathbb{E}\left(\sup_{0\leq u\leq t}|Y^{i}(u-\rho)-Y^{i,N}(u-\rho)|^{p}\right)\right]^{(p-\varepsilon)/p}
\\&+C\mathbb{E}\int_{0}^{t}\sum_{v=1}^{r}\mathbb{W}_{p}^{p}(\mathbb{L}_{Y^{i}(s+\bar{s}_{v})},\mathbb{L}_{Y^{N}(s+\bar{s}_{v})})ds.
\end{split}
\end{equation}
For $u\in[0,\rho]$, we get from \eqref{*2} that
\begin{equation*}
\mathbb{E}\left(\sup_{0\leq u\leq \rho}|Y^{i}(u)-Y^{i,N}(u)|^{p}\right)\leq C\mathbb{E}\int_{0}^{\rho}\sum_{v=1}^{r}\mathbb{W}_{p}^{p}(\mathbb{L}_{Y^{i}(s+\bar{s}_{v})},\mathbb{L}_{Y^{N}(s+\bar{s}_{v})})ds.
\end{equation*}
To deal with the Wassertein distance, for $v\in \mathbb{S}_r$ and $t\in [0,T]$, we give the definition $\mathbb{L}_{Y^{*,N}(t+\bar{s}_{v})}(\cdot)$ by
$$\mathbb{L}_{Y^{*,N}(t+\bar{s}_{v})}(\cdot)=\frac{1}{N}\sum_{j=1}^{N}\delta_{Y^{j}(t+\bar{s}_{v})}(\cdot).$$
One can observe that, for $v\in \mathbb{S}_r$ and $s\in [0,\rho]$,
\begin{equation*}
\begin{split}
&\mathbb{W}_{p}^{p}(\mathbb{L}_{Y^{i}(s+\bar{s}_{v})},\mathbb{L}_{Y^{N}(s+\bar{s}_{v})})
\\&\leq C\mathbb{W}_{p}^{p}(\mathbb{L}_{Y^{i}(s+\bar{s}_{v})},\mathbb{L}_{Y^{*,N}(s+\bar{s}_{v})})
+\mathbb{W}_{p}^{p}(\mathbb{L}_{Y^{*,N}(s+\bar{s}_{v})},\mathbb{L}_{Y^{N}(s+\bar{s}_{v})})
\\&\leq C\mathbb{W}_{p}^{p}(\mathbb{L}_{Y^{i}(s+\bar{s}_{v})},\mathbb{L}_{Y^{*,N}(s+\bar{s}_{v})})
+C\frac{1}{N}\sum_{j=1}^{N}\left|Y^{j}(s+\bar{s}_{v})-Y^{j,N}(s+\bar{s}_{v})\right|^{p}.
\end{split}
\end{equation*}
Since all $j$ are identically distributed, we have
$$\mathbb{E}\left(\frac{1}{N}\sum_{j=1}^{N}\left|Y^{j}(s+\bar{s}_{v})-Y^{j,N}(s+\bar{s}_{v})\right|^{p}\right)
=\mathbb{E}\left|Y^{i}(s+\bar{s}_{v})-Y^{i,N}(s+\bar{s}_{v})\right|^{p}.$$
Thus,
\begin{equation*}
\begin{split}
&\mathbb{E}\left(\sup_{0\leq u\leq \rho}|Y^{i}(u)-Y^{i,N}(u)|^{p}\right)
\\&\leq C\mathbb{E}\int_{0}^{\rho}\sum_{v=1}^{r}\left|Y^{i}(s+\bar{s}_{v})-Y^{i,N}(s+\bar{s}_{v})\right|^{p}ds
\\&~~~+C\mathbb{E}\int_{0}^{\rho}\mathbb{W}_{p}^{p}(\mathbb{L}_{Y^{i}(s+\bar{s}_{v})},\mathbb{L}_{Y^{*,N}(s+\bar{s}_{v})})ds
\\&\leq C\int_{0}^{\rho}\mathbb{E}\left(\sup_{0\leq u\leq s}|Y^{i}(u)-Y^{i,N}(u)|^{p}\right)ds
\\&~~~+C\mathbb{E}\int_{0}^{\rho}\mathbb{W}_{p}^{p}(\mathbb{L}_{Y^{i}(s+\bar{s}_{v})},\mathbb{L}_{Y^{*,N}(s+\bar{s}_{v})})ds.
\end{split}
\end{equation*}
Applying  Gronwall's inequality and Corollary \ref{cor2.2} yields that
\begin{equation*}
\begin{split}
\mathbb{E}\left(\sup_{0\leq u\leq \rho}|Y^{i}(u)-Y^{i,N}(u)|^{p}\right)\leq& C
\left\{\begin{array}{ll}
N^{-1 / 2}, & \text { if } p>d/2, \\
N^{-1 / 2} \log(1+N), & \text { if } p=d/2,\\
N^{-p / d}, & \text { if }2\leq p<d/2.
\end{array}\right.
\end{split}
\end{equation*}
For $u\in[0,2\rho]$, using H\"{o}lder's inequality and \eqref{*2} gives that
\begin{equation*}
\begin{split}
\mathbb{E}&\left(\sup_{0\leq u\leq 2\rho}|Y^{i}(u)-Y^{i,N}(u)|^{p}\right)
\\&\leq C\left[\mathbb{E}\left(\sup_{0\leq u\leq 2\rho}|Y^{i}(u-\rho)-Y^{i,N}(u-\rho)|^{p}\right)\right]^{(p-\varepsilon)/p}
\\&+C\mathbb{E}\int_{0}^{2\rho}\sum_{v=1}^{r}\mathbb{W}_{p}^{p}(\mathbb{L}_{Y^{i}(s+\bar{s}_{v})},\mathbb{L}_{Y^{N}(s+\bar{s}_{v})})ds
\\&\leq C\left[\mathbb{E}\left(\sup_{0\leq u\leq \rho}|Y^{i}(u)-Y^{i,N}(u)|^{p}\right)\right]^{(p-\varepsilon)/p}
\\&+C\mathbb{E}\int_{0}^{2\rho}\sum_{v=1}^{r}\mathbb{W}_{p}^{p}(\mathbb{L}_{Y^{i}(s+\bar{s}_{v})},\mathbb{L}_{Y^{*,N}(s+\bar{s}_{v})})ds
\\&+C\int_{0}^{2\rho}\mathbb{E}\left(\sup_{0\leq u\leq s}|Y^{i}(u)-Y^{i,N}(u)|^{p}\right)ds.
\end{split}
\end{equation*}
The  Gronwall inequality means that
\begin{equation*}
\begin{split}
\mathbb{E}&\left(\sup_{0\leq u\leq 2\rho}|Y^{i}(u)-Y^{i,N}(u)|^{p}\right)
\\&\leq C\left[\mathbb{E}\left(\sup_{0\leq u\leq \rho}|Y^{i}(u)-Y^{i,N}(u)|^{p}\right)\right]^{(p-\varepsilon)/p}ds
\\&+C\mathbb{E}\int_{0}^{2\rho}\sum_{v=1}^{r}\mathbb{W}_{p}^{p}(\mathbb{L}_{Y^{i}(s+\bar{s}_{v})},\mathbb{L}_{Y^{*,N}(s+\bar{s}_{v})})ds
\\&\leq C\left\{\begin{array}{ll}
(N^{-1 / 2})^{(p-\varepsilon)/p}, & \text { if } p>d/2, \\
{[N^{-1 / 2} \log(1+N)]^{(p-\varepsilon)/p}}, & \text { if } p=d/2,\\
{(N^{-p / d})^{(p-\varepsilon)/p}}, & \text { if }2\leq p<d/2.
\end{array}\right.
\end{split}
\end{equation*}
For $u\in[0,3\rho]$, we can similarly get that
\begin{equation*}
\begin{split}
\mathbb{E}&\left(\sup_{0\leq u\leq 3\rho}|Y^{i}(u)-Y^{i,N}(u)|^{p}\right)
\\&\leq C\left[\mathbb{E}\left(\sup_{0\leq u\leq 2\rho}|Y^{i}(u)-Y^{i,N}(u)|^{p}\right)\right]^{(p-\varepsilon)/p}ds
\\&+C\mathbb{E}\int_{0}^{3\rho}\sum_{v=1}^{r}\mathbb{W}_{p}^{p}(\mathbb{L}_{Y^{i}(s+\bar{s}_{v})},\mathbb{L}_{Y^{*,N}(s+\bar{s}_{v})})ds
\\&\leq C\left\{\begin{array}{ll}
(N^{-1 / 2})^{(\frac{p-\varepsilon}{p})^2}, & \text { if } p>d/2, \\
{[N^{-1 / 2} \log(1+N)]^{(\frac{p-\varepsilon}{p})^2}}, & \text { if } p=d/2,\\
{(N^{-p / d})^{(\frac{p-\varepsilon}{p})^2}}, & \text { if }2\leq p<d/2.
\end{array}\right.
\end{split}
\end{equation*}
Repeating the same procedures, we obtain that
\begin{equation*}
\begin{split}
\mathbb{E}&\left(\sup_{0\leq u\leq [(\lfloor\frac{T}{\rho}\rfloor+1)\rho]}|Y^{i}(u)-Y^{i,N}(u)|^{p}\right)
\\&\leq C\left\{\begin{array}{ll}
(N^{-1 / 2})^{\lambda_{T,\rho,p}}, & \text { if } p>d/2, \\
{[N^{-1 / 2} \log(1+N)]^{\lambda_{T,\rho,p}}}, & \text { if } p=d/2,\\
{(N^{-p / d})^{\lambda_{T,\rho,p}}}, & \text { if }2\leq p<d/2,
\end{array}\right.
\end{split}
\end{equation*}
where $\lambda_{T,\rho,p}=\left(\frac{p-\varepsilon}{p}\right)^{\lfloor\frac{T}{\rho}\rfloor}$ for $\varepsilon\in(0,1]$.
\end{proof}

\begin{rem}
From Theorem \ref{thm4.1}, we know that the value of $\varepsilon$ influences the rate of convergence
of $Y^{i,N}(\cdot)$ to $Y^{i}(\cdot)$. If the value of $\varepsilon$ is close to 0, the convergence rate
will become larger but $\bar{p}$ needs to be relatively large to make
$(l_{U}p+\varepsilon)p<\varepsilon\bar{p}$ hold. On the contrary, if the value of
$\varepsilon$ is close to 1, the requirement for $\bar{p}$ is not strict but the convergence rate
will become smaller.
\end{rem}

If we impose stronger conditions on delay components, the following theorem reveals the corresponding propagation of chaos.
\begin{thm}\label{thm4.2}
Let Assumptions  \ref{Halpha}-\ref{HD} hold with $U_{2}(x_{r},y_{r})=1$,
$K_3U_{3}(x_{r},y_{r})=K_{D}\in(0,1)$. Then there exists a constant $C$ independent of
$N$ such that, for any $i\in \mathbb{S}_N$ and $2\leq p< \bar{p}$,
\begin{equation*}
\begin{split}
\mathbb{E}&\left(\sup_{0\leq t\leq T}|Y^{i}(t)-Y^{i,N}(t)|^{p}\right)
\\&\leq C\left\{\begin{array}{ll}
N^{-1 / 2}+N^{-(\bar{p}-p)/\bar{p}}, & \text { if } p>d/2\text{ and }\bar{p}\neq2p, \\
N^{-1 / 2} \log(1+N)+N^{-(\bar{p}-p)/\bar{p}}, & \text { if } p=d/2\text{ and }\bar{p}\neq2p,\\
N^{-p / d}+N^{-(\bar{p}-p)/\bar{p}}, & \text { if }2\leq p<d/2.
\end{array}\right.
\end{split}
\end{equation*}
\end{thm}
\begin{proof}
Let the notations in this proof be the same as these in Theorem \ref{thm4.1}. We only show the main differences of the proof but omit the
same procedures. From the proof of Theorem \ref{thm4.1}, we know that
\begin{equation*}
\begin{split}
\mathbb{E}&\left[\sup_{0\leq s\leq t}(J_{1}^i(s)+J_{2}^i(s))\right]\\\leq& C\mathbb{E}\int_{0}^{t}\left|\Xi^{i}(s)\right|^{p-2}
\sum_{v=1}^{r}\left[|Y^{i}(s+\bar{s}_{v})-Y^{i,N}(s+\bar{s}_{v})|^{2}\right.
\\&\left.+\mathbb{W}_{2}^{2}(\mathbb{L}_{Y^{i}(s+\bar{s}_{v})},\mathbb{L}_{Y^{N}(s+\bar{s}_{v})})\right]ds
\\\leq& C\int_{0}^{t}\mathbb{E}|\Xi^{i}(s)|^{p}ds
+C\int_{0}^{t}\mathbb{E}\left(\sup_{0\leq u\leq s}|Y^{i}(u)-Y^{i,N}(u)|^{p}\right)ds
\\&+C\mathbb{E}\int_{0}^{t}\sum_{v=1}^{r}\mathbb{W}_{p}^{p}(\mathbb{L}_{Y^{i}(s+\bar{s}_{v})},\mathbb{L}_{Y^{N}(s+\bar{s}_{v})})ds.
\end{split}
\end{equation*}
This result with the estimation of $J_{3}^i(t)$ in the proof of Theorem \ref{thm4.1} gives that
\begin{equation}\label{4.21}
\begin{split}
\mathbb{E}&\left(\sup_{0\leq s\leq t}|\Xi^{i}(s)|^{p}\right)\\\leq&C\int_{0}^{t}\mathbb{E}\left(\sup_{0\leq u\leq s}|Y^{i}(u)-Y^{i,N}(u)|^{p}\right)ds
\\&+C\mathbb{E}\int_{0}^{t}\sum_{v=1}^{r}\mathbb{W}_{p}^{p}(\mathbb{L}_{Y^{i}(s+\bar{s}_{v})},\mathbb{L}_{Y^{N}(s+\bar{s}_{v})})ds,
\end{split}
\end{equation}
where the Gronwall inequality has been used.
Recall the elementary inequality
\begin{equation}\label{4.22}
|a+b|^p\leq\frac{1}{K_{D}^{p-1}}|a|^p+\frac{1}{(1-K_{D})^{p-1}}|b|^p,
\end{equation}
for $p\geq2$, $0<K_{D}<1$ and $a,b\in \mathbb{R}^{d}$.
For any $t\in [0,T]$, using \eqref{4.22} and Assumption \ref{HD} leads to
\begin{equation*}
\begin{split}
&|Y^{i}(t)-Y^{i,N}(t)|^{p}\\=&\left|\Xi^{i}(s)+D(Y^{i}(t-\rho))-D(Y^{i,N}(t-\rho))\right|^{p}
\\\leq&\frac{1}{K_{D}^{p-1}}\left|D(Y^{i}(t-\rho))-D(Y^{i,N}(t-\rho))\right|^{p}+\frac{1}{(1-K_{D})^{p-1}}|\Xi^{i}(s)|^{p}
\\\leq& K_{D}\left|Y^{i}(t-\rho)-Y^{i,N}(t-\rho)\right|^{p}+\frac{1}{(1-K_{D})^{p-1}}|\Xi^{i}(s)|^{p}.
\end{split}
\end{equation*}
Thus,
\begin{equation*}
\begin{split}
&\sup_{0\leq s\leq t}|Y^{i}(s)-Y^{i,N}(s)|^{p}
\\\leq& K_{D}\sup_{0\leq s\leq t}\left|Y^{i}(s-\rho)-Y^{i,N}(s-\rho)\right|^{p}+\frac{1}{(1-K_{D})^{p-1}}\sup_{0\leq s\leq t}|\Xi^{i}(s)|^{p}
\\\leq& \frac{1}{(1-K_{D})^{p}}\sup_{0\leq s\leq t}|\Xi^{i}(s)|^{p}.
\end{split}
\end{equation*}
By the above inequality, (\ref{4.21}), Gronwall's inequality and the technique in the proof of Theorem \ref{thm4.1}, we have
\begin{equation*}
\begin{split}
\mathbb{E}&\left(\sup_{0\leq s\leq t}|Y^{i}(s)-Y^{i,N}(s)|^{p}\right)
\\&\leq C\mathbb{E}\int_{0}^{t}\sum_{v=1}^{r}\mathbb{W}_{p}^{p}(\mathbb{L}_{Y^{i}(s+\bar{s}_{v})},\mathbb{L}_{Y^{N}(s+\bar{s}_{v})})ds
\\&\leq C\int_{0}^{t}\mathbb{E}\left(\sup_{0\leq u\leq s}|Y^{i}(u)-Y^{i,N}(u)|^{p}\right)ds
\\&+C\mathbb{E}\int_{0}^{t}\sum_{v=1}^{r}\mathbb{W}_{p}^{p}(\mathbb{L}_{Y^{i}(s+\bar{s}_{v})},\mathbb{L}_{Y^{*,N}(s+\bar{s}_{v})})ds.
\end{split}
\end{equation*}
Applying the Gronwall inequality and Theorem \ref{thm2.1} leads to
\begin{equation*}
\begin{split}
\mathbb{E}&\left(\sup_{0\leq s\leq t}|Y^{i}(s)-Y^{i,N}(s)|^{p}\right)
\\&\leq C\mathbb{E}\int_{0}^{t}\sum_{v=1}^{r}\mathbb{W}_{p}^{p}(\mathbb{L}_{Y^{i}(s+\bar{s}_{v})},\mathbb{L}_{Y^{*,N}(s+\bar{s}_{v})})ds
\\&\leq C\left\{\begin{array}{ll}
N^{-1 / 2}+N^{-(\bar{p}-p)/\bar{p}}, & \text { if } p>d/2\text{ and }\bar{p}\neq2p, \\
N^{-1 / 2} \log(1+N)+N^{-(\bar{p}-p)/\bar{p}}, & \text { if } p=d/2\text{ and }\bar{p}\neq2p,\\
N^{-p / d}+N^{-(\bar{p}-p)/\bar{p}}, & \text { if }2\leq p<d/2.
\end{array}\right.
\end{split}
\end{equation*}
\end{proof}

\begin{rem}
By comparing Theorem \ref{thm4.1} and Theorem \ref{thm4.2}, we can find that: to relax the constraints of  the delay term and neutral term, we introduce the functions $U_{2}(\cdot,\cdot)$ and $U_{3}(\cdot,\cdot)$, which makes the convergence rate of $Y^{i,N}(\cdot)$ to $Y^{i}(\cdot)$ less ideal.
\end{rem}
\begin{rem}
	The techniques in Theorem \ref{thm4.1} and Theorem \ref{thm4.2} can be used in the theory about the propagation of chaos in $\mathcal{L}^p$ sense.
\end{rem}

\section{Numerical scheme for NMSMVEs}\label{section5}
In this section, we establish the tamed EM method for superlinear NMSMVEs. 
Firstly, the boundedness of the numerical solutions is analyzed. 
Then the strong convergence rate is obtained by using propagation of chaos.
To investigate the tamed EM scheme for \eqref{4.2},  for $\Delta\in(0,1)$ and $\gamma\in (0,1/2]$, define
\begin{equation}\label{5.1}
\alpha_{\Delta}(\Gamma(\varphi),\mathbb{L}_{\Gamma(\varphi)})
=\frac{\alpha(\Gamma(\varphi),\mathbb{L}_{\Gamma(\varphi)})}{1+\Delta^{\gamma}|\alpha(\Gamma(\varphi),\mathbb{L}_{\Gamma(\varphi)})|},
\end{equation}
where $\Gamma(\varphi)\in (\mathbb{R}^{d})^{r}$, $\mathbb{L}_{\Gamma(\varphi)}\in (\mathcal{P}_{\bar{p}}(\mathbb{R}^{d}))^{r}$.
Assume that there exist two positive integers $M$ and $M_{T}$ such that $\Delta=\frac{\rho}{M}=\frac{T}{M_{T}}$ and other positive integers $\bar{k}_{2},\bar{k}_{3},\cdots,\bar{k}_{r-1}$ such that $\frac{-\bar{s}_{2}}{\bar{k}_{2}}=\frac{-\bar{s}_{3}}{\bar{k}_{3}}=\cdots=\frac{-\bar{s}_{r-1}}{\bar{k}_{r-1}}=\Delta.$
Set $t_{k}=k\Delta$, $k=-M,\ldots,0,1,\ldots,M_{T}$. Define the tamed EM scheme as:
\begin{equation*}
\begin{split}
	\left\{\begin{array}{lll}
		&X^{i,N}(t_{k})=\xi^{i}(t_{k}), ~~~ k=-M,-M+1,\ldots,0, \\
		&X^{i,N}(t_{k+1})-D(X^{i,N}(t_{k+1-M}))
		\\&=X^{i,N}(t_{k})-D(X^{i,N}(t_{k-M}))+\alpha_{\Delta}(\Gamma(X^{i,N}_{t_{k}}),\mathbb{L}_{\Gamma(X^{N}_{t_{k}})})\Delta
		\\&~~~+\beta(\Gamma(X^{i,N}_{t_{k}}),\mathbb{L}_{\Gamma(X^{N}_{t_{k}})})\Delta B^{i}_{k},~~~k=0,1,\ldots,M_{T}-1,
	\end{array}\right.
\end{split}
\end{equation*}
where 
\begin{equation*}
\Delta B^{i}_{k}=B^{i}(t_{k+1})-B^{i}(t_{k}),
\end{equation*}
\begin{equation*}
\begin{split}
\Gamma(X^{i,N}_{t_{k}})=\left(X^{i,N}(t_{k}),X^{i,N}(t_{k-\bar{k}_{2}}),\cdots,X^{i,N}(t_{k-M})\right),
\end{split}
\end{equation*}
\begin{equation*}
\begin{split}
\mathbb{L}_{\Gamma(X^{N}_{t_{k}})}=\left(\mathbb{L}_{X^{N}(t_{k})},\mathbb{L}_{X^{N}(t_{k-\bar{k}_{2}})},\cdots,\mathbb{L}_{X^{N}(t_{k-M})}\right),
\end{split}
\end{equation*}
and
 \begin{equation*}
 	\mathbb{L}_{X^{N}(t_{k-\bar{k}_{v}})}(\cdot)
=\frac{1}{N}\sum_{j=1}^{N}\delta_{X^{j,N}(t_{k-\bar{k}_{v}})}(\cdot),~~~v\in \mathbb{S}_r.
\end{equation*}

\begin{rem}
In this discrete-time numerical scheme, for any $i\in \mathbb{S}_{N}$, we use
$\big(X^{i,N}(t_{k}), X^{i,N}(t_{k-\bar{k}_{2}}), \cdots, X^{i,N}(t_{k-M})\big)$
to approximate \\
$\big(Y^{i,N}(t_{k}),Y^{i,N}(t_{k-\bar{k}_{2}}),\cdots,Y^{i,N}(t_{k-M})\big).$
This can be achieved in the numerical simulation. The reason is that: for the largest delay constant $\rho$, we have assumed that there exist a positive integer $M$ such that $\Delta=\frac{\rho}{M}$, which means that there exist $M-1$ time grids to divide the delay interval into $M$ parts. Then other delays may be at these time grids, otherwise, we can make $M$ larger to achieve this goal. Ideally, we have $\frac{-\bar{s}_{2}}{\bar{k}_{2}}=\frac{-\bar{s}_{3}}{\bar{k}_{3}}=\cdots=\frac{-\bar{s}_{r-1}}{\bar{k}_{r-1}}=\frac{\rho}{M}=\Delta$for the most appropriate $M$.
\end{rem}

The continuous-time step process numerical solution on $t\in[-\rho,T]$ is defined by
\begin{equation*}
	\begin{split}
		&X^{i,N}(t)=\sum_{k=-M}^{M_T}X^{i,N}(t_{k})\mathbb{I}_{[t_k,t_{k+1})}(t).
	\end{split}
\end{equation*}
For $t\in[0,T]$, the continuous-sample numerical solution is defined by
\begin{equation}\label{contin}
\begin{split}
&\bar{X}^{i,N}(t)-D(\bar{X}^{i,N}(t-\rho))
\\&=\xi^{i}(0)-D(\xi^{i}(-\rho))+\int_0^t\alpha_{\Delta}(\Gamma(X^{i,N}_{s}),\mathbb{L}_{\Gamma(X^{N}_{s})})ds
\\&+\int_0^t\beta(\Gamma(X^{i,N}_{s}),\mathbb{L}_{\Gamma(X^{N}_{s})})d B^{i}(s),
\end{split}
\end{equation}
where $\Gamma(X^{i,N}_{t})=\Gamma(X^{i,N}_{t_{k}})$ for $t\in[t_k,t_{k+1})$ and $\mathbb{L}_{\Gamma(X^{N}_{t})}(\cdot)=\frac{1}{N}\sum_{j=1}^{N}\delta_{\Gamma(X^{j,N}_{t})}(\cdot)$.
We can find that $\mathbb{L}_{\Gamma(X^{N}_{t})}=\mathbb{L}_{\Gamma(X^{N}_{t_{k}})}$ for any $t\in[t_k,t_{k+1})$. 
Similar to \cite{21,36}, we know that $\bar{X}^{i,N}(t_{k})=X^{i,N}(t_{k})=X^{i,N}(t)$ for any $t\in[t_k,t_{k+1}).$
We get from (\ref{5.1}) that
\begin{equation}\label{tedi1}
|\alpha_{\Delta}(\Gamma(\varphi),\mathbb{L}_{\Gamma(\varphi)})|\leq\Delta^{-\gamma}\wedge|\alpha(\Gamma(\varphi),\mathbb{L}_{\Gamma(\varphi)})|,
\end{equation}
for any  $\Gamma(\varphi)\in (\mathbb{R}^{d})^{r}$, $\mathbb{L}_{\Gamma(\varphi)}\in (\mathcal{P}_{\bar{p}}(\mathbb{R}^{d}))^{r}$.
By (\ref{tuichu2}), we derive that
\begin{equation}\label{tedi2}
\begin{split}
&(x_{1}-D(x_{r}))^{T}\alpha_{\Delta}(x^{(r)},\mu^{(r)})
\\&\leq C\left[1+\sum_{i=1}^{r-1}|x_{i}|^2+U^{2}_{2}(x_{r},0)|x_{r}|^2+\sum_{i=1}^{r}\mathcal{W}_{2}^{2}(\mu_i)\right],
\end{split}
\end{equation}
for any $x^{(r)}\in (\mathbb{R}^{d})^{r}$ and $\mu^{(r)}\in (\mathcal{P}_{2}(\mathbb{R}^{d}))^{r}$.
For simplicity, define $t_{\Delta}=\lfloor\frac{t}{\Delta}\rfloor\Delta$ for any $t\in [-\rho,T]$.

	\begin{lem}\label{xinxinlem5.1}
	Let Assumptions \ref{Halpha}-\ref{HD} hold. Then for any $\tilde{p}>0$, we have
	\begin{equation*}
		\max_{i\in \mathbb{S}_N}\sup_{0< \Delta\leq 1}\sup_{0\leq t\leq T}\mathbb{E}|\bar{X}^{i,N}(t)|^{\tilde{p}}
		\leq C, \quad \forall T>0.
	\end{equation*}
\end{lem}
\begin{proof}
	Let $\tilde{p}\geq 4$ first. Applying It\^{o}'s formula gives that
	\begin{equation*}
		\begin{split}
			&\left|\bar{X}^{i,N}(t)-D(\bar{X}^{i,N}(t-\rho))\right|^{\tilde{p}}-\left|\xi^{i}(0)-D(\xi^{i}(-\rho))\right|^{\tilde{p}}
			\\&\leq\tilde{p}\int_{0}^{t}\left|\bar{X}^{i,N}(s)-D(\bar{X}^{i,N}(s-\rho))\right|^{\tilde{p}-2}
			\left(\bar{X}^{i,N}(s)-D(\bar{X}^{i,N}(s-\rho))\right)^{T}\\&\quad \quad \quad\cdot\alpha_{\Delta}\left(\Gamma({X}^{i,N}_{s}),
			\mathbb{L}_{\Gamma({X}^{N}_{s})}\right)ds
			\\&+\frac{\tilde{p}(\tilde{p}-1)}{2}\int_{0}^{t}\left|\bar{X}^{i,N}(s)-D(\bar{X}^{i,N}(s-\rho))\right|^{\tilde{p}-2}
			\left|\beta\left(\Gamma({X}^{i,N}_{s}),
			\mathbb{L}_{\Gamma({X}^{N}_{s})}\right)\right|^{2}ds
			\\&+\tilde{p}\int_{0}^{t}|\bar{X}^{i,N}(s)-D(\bar{X}^{i,N}(s-\rho))|^{\tilde{p}-2}(\bar{X}^{i,N}(s)-D(\bar{X}^{i,N}(s-\rho)))^{T}
			\\&\quad \quad \quad\beta\left(\Gamma({X}^{i,N}_{s}),
			\mathbb{L}_{\Gamma({X}^{N}_{s})}\right)dB^{i}(s)
			\\&=:Q_{1}^{i,N}(t)+Q_{2}^{i,N}(t)+Q_{3}^{i,N}(t).
		\end{split}
	\end{equation*}
Moreover,
	\begin{equation*}
		\begin{split}
			&Q_{1}^{i,N}(t)\\&= C\int_{0}^{t}\left|\bar{X}^{i,N}(s)-D(\bar{X}^{i,N}(s-\rho))\right|^{\tilde{p}-2}
			\\& \cdot\left(\bar{X}^{i,N}(s)-D(\bar{X}^{i,N}(s-\rho))-X^{i,N}(s)+D(X^{i,N}(s-\rho))\right)^{T}
		\alpha_{\Delta}\left(\Gamma({X}^{i,N}_{s}),
			\mathbb{L}_{\Gamma({X}^{N}_{s})}\right)ds
			\\&+C\int_{0}^{t}\left|\bar{X}^{i,N}(s)-D(\bar{X}^{i,N}(s-\rho))\right|^{\tilde{p}-2}
			\\&\cdot \left(X^{i,N}(s)-D(X^{i,N}(s-\rho))\right)^{T}\alpha_{\Delta}\left(\Gamma(X^{i,N}_{s}),
			\mathbb{L}_{\Gamma(X^{N}_{s})}\right)ds.
			\\&=:Q_{11}^{i,N}(t)+Q_{12}^{i,N}(t).
		\end{split}
	\end{equation*}
	By (\ref{contin}), Young's inequality, H\"{o}lder's inequality and Assumptions \ref{Halpha}, \ref{Hbeta}, we have
	\begin{equation*}
		\begin{split}
			&\mathbb{E}Q_{11}^{i,N}(t)
			\\&\leq C\mathbb{E}\int_{0}^{t}\left|\bar{X}^{i,N}(s)-D(\bar{X}^{i,N}(s-\rho))\right|^{\tilde{p}-2}
			\left|\int_{s_{\Delta}}^{s}\alpha_{\Delta}\left(\Gamma(X^{i,N}_{u}),\mathbb{L}_{\Gamma(X^{N}_{u})}\right)du\right.
			\\&\quad \left.+\int_{s_{\Delta}}^{s}\beta\left(\Gamma(X^{i,N}_{u}),\mathbb{L}_{\Gamma(X^{N}_{u})}\right)dB^{i}(u)\right|
			\left|\alpha_{\Delta}\left(\Gamma(X^{i,N}_{s}),\mathbb{L}_{\Gamma(X^{N}_{s})}\right)\right|ds.
			\\&\leq C\mathbb{E}\int_{0}^{t}\left|\bar{X}^{i,N}(s)-D(\bar{X}^{i,N}(s-\rho))\right|^{\tilde{p}}ds
			\\&\quad +C\Delta^{-\gamma\tilde{p}/2}\mathbb{E}\int_{0}^{t}\left(\left|\int_{s_{\Delta}}^{s}\alpha_{\Delta}\left(\Gamma(X^{i,N}_{u}),
			\mathbb{L}_{\Gamma(X^{N}_{u})}\right)du\right|^{\tilde{p}/2}\right.
			\\&\quad\left.+\left|\int_{s_{\Delta}}^{s}\beta\left(\Gamma(X^{i,N}_{u}),\mathbb{L}_{\Gamma(X^{N}_{u})}\right)dB^{i}(u)\right|^{\tilde{p}/2} \right)ds
			\\&\leq C\mathbb{E}\int_{0}^{t}\left|\bar{X}^{i,N}(s)-D(\bar{X}^{i,N}(s-\rho))\right|^{\tilde{p}}ds
			\\&\quad +C\left(\Delta^{(1/2-\gamma)\tilde{p}}+\Delta^{-\gamma\tilde{p}/2}\mathbb{E}\int_{0}^{t}
			\left|\int_{s_{\Delta}}^{s}\beta\left(\Gamma(X^{i,N}_{u}),\mathbb{L}_{\Gamma(X^{N}_{u})}\right)dB^{i}(u)\right|^{\tilde{p}/2}ds\right)
			\\&\leq C\mathbb{E}\int_{0}^{t}\left|\bar{X}^{i,N}(s)-D(\bar{X}^{i,N}(s-\rho))\right|^{\tilde{p}}ds
			\\&\quad +C\Delta^{(1/2-\gamma)\tilde{p}}+C\Delta^{-\gamma\tilde{p}/2}\int_{0}^{t}
			\mathbb{E}\left(\int_{s_{\Delta}}^{s}\left|\beta\left(\Gamma(X^{i,N}_{u}),
			\mathbb{L}_{\Gamma(X^{N}_{u})}\right)\right|^{2}du\right)^{\tilde{p}/4}ds
			\\
				\end{split}
		\end{equation*}
			\begin{equation*}
				\begin{split}
			&\leq C\mathbb{E}\int_{0}^{t}\left|\bar{X}^{i,N}(s)-D(\bar{X}^{i,N}(s-\rho))\right|^{\tilde{p}}ds
			+C\Delta^{(1/2-\gamma)\tilde{p}}
			\\&\quad +C\Delta^{(1/2-\gamma)\tilde{p}/2}\int_{0}^{t}\mathbb{E}\left[1
			+\sum_{v=1}^{r-1}|X^{i,N}(s+\bar{s}_{v})|^{\tilde{p}/2}\right.
			\\&\quad\left.+U^{\tilde{p}/2}_{2}(X^{i,N}(s-\rho),0)|X^{i,N}(s-\rho)|^{\tilde{p}/2}
			+\sum_{i=1}^{r}\mathcal{W}_{2}^{\tilde{p}/2}(\mathbb{L}_{X^{N}(s+\bar{s}_{v})})\right]ds
			\\&\leq C\mathbb{E}\int_{0}^{t}\left|\bar{X}^{i,N}(s)-D(\bar{X}^{i,N}(s-\rho))\right|^{\tilde{p}}ds+C\left(1+\int_{0}^{t}\mathbb{E}\left[\sum_{v=1}^{r-1}|X^{i,N}(s+\bar{s}_{v})|^{\tilde{p}}\right.\right.
			\\&\quad\left.\left. +|X^{i,N}(s-\rho)|^{(l_U+1)\tilde{p}}
			+\sum_{i=1}^{r}\mathcal{W}_{\tilde{p}}^{\tilde{p}}(\mathbb{L}_{X^{N}(s+\bar{s}_{v})})\right]ds\right)
			\\&\leq C\mathbb{E}\int_{0}^{t}\left|\bar{X}^{i,N}(s)-D(\bar{X}^{i,N}(s-\rho))\right|^{\tilde{p}}ds
			+C\left(1+\int_{0}^{t}\sup_{0\leq u\leq s}\mathbb{E}\left|\bar{X}^{i,N}(u)\right|^{\tilde{p}}ds\right)
			\\&\quad +C\int_{0}^{t}\mathbb{E}|X^{i,N}(s-\rho)|^{(l_U+1)\tilde{p}}ds.
		\end{split}
	\end{equation*}
	By  H\"{o}lder's inequality, Young's inequality, (\ref{tedi2}) and Assumptions \ref{Halpha}-\ref{HD}, one can see that
	\begin{equation*}
		\begin{split}
			&\mathbb{E}\left(Q_{12}^{i,N}(t)+Q_{2}^{i,N}(t)\right)
			\\&\leq C\mathbb{E}\int_{0}^{t}\left|\bar{X}^{i,N}(s)-D(\bar{X}^{i,N}(s-\rho))\right|^{\tilde{p}-2}\left[1
			+\sum_{v=1}^{r-1}|X^{i,N}(s+\bar{s}_{v})|^2\right.
			\\&\quad\left.+U^{2}_{2}(X^{i,N}(s-\rho),0)|X^{i,N}(s-\rho)|^2
			+\sum_{i=1}^{r}\mathcal{W}_{2}^{2}(\mathbb{L}_{X^{N}(s+\bar{s}_{v})})\right]ds
			\\&\leq C\mathbb{E}\int_{0}^{t}\left[1+|X^{i,N}(s)|^{\tilde{p}}
			+U^{\tilde{p}}_{3}(\bar{X}^{i,N}(s-\rho),0)|\bar{X}^{i,N}(s-\rho)|^{\tilde{p}}\right.
			\\&\quad\left.+\sum_{v=1}^{r-1}|X^{i,N}(s+\bar{s}_{v})|^{\tilde{p}}\right]ds+U^{\tilde{p}}_{2}(X^{i,N}(s-\rho),0)|X^{i,N}(s-\rho)|^{\tilde{p}}
			\\&\quad\left.+\sum_{v=1}^{r}\mathcal{W}_{\tilde{p}}^{2}(\mathbb{L}_{X^{N}(s+\bar{s}_{v})})\right]ds
			\\&\leq C\mathbb{E}\int_{0}^{t}\left[1+|X^{i,N}(s)|^{\tilde{p}}
			+U^{2\tilde{p}}_{3}(\bar{X}^{i,N}(s-\rho),0)+|\bar{X}^{i,N}(s-\rho)|^{2\tilde{p}}\right.
			\\&\quad+\sum_{v=1}^{r-1}|X^{i,N}(s+\bar{s}_{v})|^{\tilde{p}}+U^{2\tilde{p}}_{2}(X^{i,N}(s-\rho),0)+|X^{i,N}(s-\rho)|^{2\tilde{p}}
			\\&\quad\left.+\sum_{v=1}^{r}\mathcal{W}_{\tilde{p}}^{2}(\mathbb{L}_{X^{N}(s+\bar{s}_{v})})\right]ds
			\\
		\end{split}
	\end{equation*}
			\begin{equation*}
				\begin{split}
			&\leq C\int_{0}^{t}\left[1+\sup_{0\leq u\leq s}\mathbb{E}\left|\bar{X}^{i,N}(u)\right|^{\tilde{p}}\right.
			\\&\quad\left.+|\bar{X}^{i,N}(s-\rho)|^{2\tilde{p}l_U}+|\bar{X}^{i,N}(s-\rho)|^{2\tilde{p}}
			+|X^{i,N}(s-\rho)|^{2\tilde{p}l_U}+|X^{i,N}(s-\rho)|^{2\tilde{p}}\right]
			\\&\leq C\int_{0}^{t}\left[1+\sup_{0\leq u\leq s}\mathbb{E}\left|\bar{X}^{i,N}(u)\right|^{\tilde{p}}\right]ds
			\\&\quad +C\int_{0}^{t}\mathbb{E}\left(|\bar{X}^{i,N}(s-\rho)|^{\tilde{p}l_{U}^{*}}+X^{i,N}(s-\rho)|^{\tilde{p}l_{U}^{*}}\right)ds.
		\end{split}
	\end{equation*}
	where $l_{U}^{*}=2l_{U}+2$.
	Combining these inequalities with the Gronwall inequality leads to
	\begin{equation*}
		\begin{split}
			&\sup_{0\leq s\leq t}\mathbb{E}\left|\bar{X}^{i,N}(s)-D(\bar{X}^{i,N}(s-\rho))\right|^{\tilde{p}}
			\\&\leq C\left(1+\int_{0}^{t}\sup_{0\leq u\leq s}\mathbb{E}|\bar{X}^{i,N}(u)|^{\tilde{p}}ds
		 +\int_{0}^{t}\mathbb{E}\left(|\bar{X}^{i,N}(s-\rho)|^{\tilde{p}l_{U}^{*}}+X^{i,N}(s-\rho)|^{\tilde{p}l_{U}^{*}}\right)ds\right).
		\end{split}
	\end{equation*}
	Thus, we can get immediately that
	\begin{equation*}
		\begin{split}
			&\sup_{0\leq s\leq t}\mathbb{E}|\bar{X}^{i,N}(s)|^{\tilde{p}}
			\\&\leq C\left(\sup_{0\leq s\leq t}\mathbb{E}\left|\bar{X}^{i,N}(s)-D(\bar{X}^{i,N}(s-\rho))\right|^{\tilde{p}}
			+\sup_{0\leq s\leq t}\mathbb{E}|D(\bar{X}^{i,N}(s-\rho))|^{\tilde{p}}\right)
			\\&\leq C\left(1+\int_{0}^{t}\sup_{0\leq u\leq s}\mathbb{E}|\bar{X}^{i,N}(u)|^{\tilde{p}}ds
			+\sup_{0\leq s\leq t}\mathbb{E}\left|\bar{X}^{i,N}(s-\rho)\right|^{\tilde{p}l_{U}^{*}}\right).
		\end{split}
	\end{equation*}
	Thanks to Gronwall's inequality, one can see that
	\begin{equation*}
		\sup_{0\leq s\leq t}\mathbb{E}|\bar{X}^{i,N}(s)|^{\tilde{p}}
		\leq C\left(1+\sup_{0\leq s\leq t}\mathbb{E}\left|\bar{X}^{i,N}(s-\rho)\right|^{\tilde{p}l_{U}^{*}}\right).
	\end{equation*}
	By recalling the proof of Theorem \ref{thm3.1}, we can construct a finite sequence $\{\tilde{p}_{1},\tilde{p}_{2},\ldots,\tilde{p}_{M_{T}+1}\}$ such that $\tilde{p}_{j+1}l_{U}^{*}<\tilde{p}_{j}$ and $\tilde{p}_{M_{T}+1}=\tilde{p}$ for $j=1,2,\ldots,M_{T}+1$. Then using the same technique means the desired result when $\tilde{p}\geq 4$. The case when $\tilde{p}\in  (0,4)$ follows from the H\"{o}lder inequality.
\end{proof}

\begin{lem}\label{lem5.1}
Let Assumptions \ref{Halpha}, \ref{Hbeta}, \ref{HD} hold. Then for any $\bar{p}\in\left[2,\frac{\tilde{p}}{2l_U+ 2}\right]$, we know that
\begin{equation*}
\max_{i\in \mathbb{N}_N}\sup_{0< \Delta\leq 1}\mathbb{E}\left(\sup_{0\leq t\leq T}|\bar{X}^{i,N}(t)|^{\bar{p}}\right)
\leq C, \quad \forall T>0.
\end{equation*}
\end{lem}
\begin{proof}
After analysis, we know that the key difference is to estimate\\ $\mathbb{E}\left(\sup_{0\leq s\leq t}Q_{3}^{i,N}(s)\right)$.
Using BDG's inequality,  Young’s inequality, H\"{o}lder’s inequality gives that
\begin{equation*}
\begin{split}
&\mathbb{E}\left(\sup_{0\leq s\leq t}Q_{3}^{i,N}(s)\right)
\\&\leq C\mathbb{E}\left[\int_{0}^{t}\left|\bar{X}^{i,N}(s)-D(\bar{X}^{i,N}(s-\rho))\right|^{2\bar{p}-2}\left|\beta\left(\Gamma(X^{i,N}_{s}),
\mathbb{L}_{\Gamma(X^{N}_{s})}\right)\right|^{2} ds\right]^{1/2}
\\&\leq C\mathbb{E}\left[\sup_{0\leq s\leq t}\left|\bar{X}^{i,N}(s)-D(\bar{X}^{i,N}(s-\rho))\right|^{\bar{p}-1}\right.
\\&\quad\quad \quad \cdot\left.\left(\int_{0}^{t}\left|\beta\left(\Gamma(X^{i,N}_{s}),
\mathbb{L}_{\Gamma(X^{N}_{s})}\right)\right|^{2} ds\right)^{1/2}\right]
\\&\leq \frac{1}{2}\mathbb{E}\left(\sup_{0\leq s\leq t}\left|\bar{X}^{i,N}(s)-D(\bar{X}^{i,N}(s-\rho))\right|^{\bar{p}}\right)
\\&+C\mathbb{E}\left[\int_{0}^{t}\left[1
+\sum_{v=1}^{r-1}|X^{i,N}(s+\bar{s}_{v})|^{2}\right.\right.
\\&\quad\left.\left.+U^{2}_{2}(X^{i,N}(s-\rho),0)|X^{i,N}(s-\rho)|^{2}
+\sum_{i=1}^{r}\mathcal{W}_{2}^{2}(\mathbb{L}_{X^{N}(s+\bar{s}_{v})})\right]ds\right]^{\bar{p}/2}
\\&\leq \frac{1}{2}\mathbb{E}\left(\sup_{0\leq s\leq t}\left|\bar{X}^{i,N}(s)-D(\bar{X}^{i,N}(s-\rho))\right|^{\bar{p}}\right)
\\&+C\int_{0}^{t}\left(1+\sup_{0\leq u\leq s}\mathbb{E}|\bar{X}^{i,N}(u)|^{\bar{p}}+\sup_{0\leq u\leq s}\mathbb{E}|\bar{X}^{i,N}(u)|^{(l_U+1)\bar{p}}\right)ds
\\&\leq \frac{1}{2}\mathbb{E}\left(\sup_{0\leq s\leq t}\left|\bar{X}^{i,N}(s)-D(\bar{X}^{i,N}(s-\rho))\right|^{\bar{p}}\right)
+C.
\end{split}
\end{equation*}
Then by the results in the proof of Lemma \ref{xinxinlem5.1}, we derive that
\begin{equation*}
\begin{split}
&\mathbb{E}\left(\sup_{0\leq s\leq t}\left|\bar{X}^{i,N}(s)-D(\bar{X}^{i,N}(s-\rho))\right|^{\bar{p}}\right)
\\&\leq C+C\int_{0}^{t}\sup_{0\leq u\leq s}\mathbb{E}|\bar{X}^{i,N}(u)|^{\bar{p}}ds
 +C\int_{0}^{t}\sup_{0\leq u\leq s}\mathbb{E}|\bar{X}^{i,N}(u-\rho)|^{\bar{p}l_{U}^{*}}ds
 \\&\leq C,
\end{split}
\end{equation*}
where $\bar{p}{l_U^*}\leq\tilde{p}$ has been used.
Thus, we can get immediately that
\begin{equation*}
\begin{split}
&\mathbb{E}\left(\sup_{0\leq s\leq t}|\bar{X}^{i,N}(s)|^{\bar{p}}\right)
\\&\leq C\mathbb{E}\left(\sup_{0\leq s\leq t}\left|\bar{X}^{i,N}(s)-D(\bar{X}^{i,N}(s-\rho))\right|^{\bar{p}}\right)
+C\mathbb{E}\left(\sup_{0\leq s\leq t}|D(\bar{X}^{i,N}(s-\rho))|^{\bar{p}}\right)
\\&\leq C
+C\mathbb{E}\left(\sup_{0\leq s\leq t}\left|\bar{X}^{i,N}(s-\rho)\right|^{\bar{p}l_{U}^{*}}\right).
\end{split}
\end{equation*}
Similar to the last part in the proof of Lemma \ref{xinxinlem5.1}, we get the desired result.
\end{proof}

\begin{lem}\label{lem5.2}
Let Assumptions \ref{Halpha}-\ref{HD} hold. Then for any $\hat{p}\in [2,\frac{\bar{p}}{2l_{U}+2}]$, it holds that
\begin{equation*}
\max_{i\in \mathbb{N}_N}\mathbb{E}\left(\sup_{0\leq k\leq M_{T}}\sup_{t_{k}\leq t\leq t_{k+1}}|\bar{X}^{i,N}(t)-X^{i,N}(t)|^{\hat{p}}\right)
\leq C\Delta^{\hat{p}/2}.
\end{equation*}
\end{lem}
\begin{proof}
By (\ref{contin}), for any $t\in[t_{k},t_{k+1})$, we have
\begin{equation*}
\bar{X}^{i,N}(t)-X^{i,N}(t)=D(\bar{X}^{i,N}(t-\rho))-D(X^{i,N}(t-\rho))+\psi(t),
\end{equation*}
where 
\begin{equation*}
	\psi(t):=\int_{t_{k}}^{t}\alpha_{\Delta}\left(\Gamma(X^{i,N}_{s}),\mathbb{L}_{\Gamma(X^{N}_{s})}\right)ds
+\int_{t_{k}}^{t}\beta\left(\Gamma(X^{i,N}_{s}),\mathbb{L}_{\Gamma(X^{N}_{s})}\right)dB^{i}(s).
\end{equation*}
By Assumption \ref{Hbeta}, Lemma \ref{lem5.1} and (\ref{tedi1}), we derive that
\begin{equation*}
\begin{split}
&\mathbb{E}\left(\sup_{t_{k}\leq t\leq t_{k+1}}|\psi(t)|^{\hat{p}}\right)
\\&\leq 2^{\hat{p}-1}\mathbb{E}\left(\sup_{t_{k}\leq t\leq t_{k+1}}\left|\int_{t_{k}}^{t}\alpha_{\Delta}\left(\Gamma(X^{i,N}_{s}),\mathbb{L}_{\Gamma(X^{N}_{s})}\right)ds\right|^{\hat{p}}\right)
\\&\quad +2^{\hat{p}-1}\mathbb{E}\left(\sup_{t_{k}\leq t\leq t_{k+1}}\left|\int_{t_{k}}^{t}\beta\left(\Gamma(X^{i,N}_{s}),\mathbb{L}_{\Gamma(X^{N}_{s})}\right)dB^{i}(s)\right|^{\hat{p}}\right)
\\&\leq C\Delta^{(1-\gamma)\hat{p}}+C\mathbb{E}\left[\sup_{t_{k}\leq t\leq t_{k+1}}\left|\beta\left(\Gamma(X^{i,N}_{t_{k}}),\mathbb{L}_{\Gamma(X^{N}_{t_{k}})}\right)\left(B^{i}(t)-B^{i}(t_{k})\right)\right|^{\hat{p}}\right]
\\&\leq C\Delta^{(1-\gamma)\hat{p}}+C\Delta^{\hat{p}/2}\mathbb{E}\left(1
+\sum_{v=1}^{r-1}|X^{i,N}(t_{k-\bar{k}_{v}})|^{\hat{p}}\right.\\
&\quad\left.+U^{\hat{p}}_{2}(X^{i,N}(t_{k-M}),0)|X^{i,N}(t_{k-M})|^{\hat{p}}+\sum_{v=1}^{r}\mathcal{W}^{\hat{p}}_{2}(\mathbb{L}_{X^{N}(t_{k-\bar{k}_{v}})})\right)
\\&\leq C\Delta^{\hat{p}/2}.
\end{split}
\end{equation*}
Thus, using Assumption \ref{HD} means that
\begin{equation*}\label{5.20}
\begin{split}
&\mathbb{E}\left(\sup_{t_{k}\leq t\leq t_{k+1}}|\bar{X}^{i,N}(t)-X^{i,N}(t)|^{\hat{p}}\right)
\\&\leq C\mathbb{E}\left(\sup_{t_{k}\leq t\leq t_{k+1}}|D(\bar{X}^{i,N}(t-\rho))-D(X^{i,N}(t-\rho))|^{\hat{p}}\right)
+C\mathbb{E}\left(\sup_{t_{k}\leq t\leq t_{k+1}}|\psi(t)|^{\hat{p}}\right)
\\&\leq C\left[\mathbb{E}\left(\sup_{t_{k}\leq t\leq t_{k+1}}U^{2\hat{p}}_{3}(\bar{X}^{i,N}(t-\rho),X^{i,N}(t-\rho))\right)\right]^{1/2}
\\&\quad\quad \quad  \left[\mathbb{E}\left(\sup_{t_{k}\leq t\leq t_{k+1}}|\bar{X}^{i,N}(t-\rho)-X^{i,N}(t-\rho)|^{2\hat{p}}\right)\right]^{1/2}+C\Delta^{\hat{p}/2}
\\
\end{split}
\end{equation*}
\begin{equation}\label{5.20}
\begin{split}
&\leq C\Delta^{\hat{p}/2}+C\left[\mathbb{E}\left(\sup_{t_{k}\leq t\leq t_{k+1}}|\bar{X}^{i,N}(t-\rho)-X^{i,N}(t-\rho)|^{2\hat{p}}\right)\right]^{1/2}.
\end{split}
\end{equation}
Define 
\begin{equation*}
\hat{p}_{j}=(M_T+2-j)\hat{p}2^{M_T+1-j},~~~j=1,2,\ldots,M_{T}+1.
\end{equation*}
One can observe that 
\begin{equation*}
2\hat{p}_{j+1}<\hat{p}_{j}~~~\text{and} ~~~\hat{p}_{M_{T}+1}=\hat{p},~~~j=1,2,\ldots,M_{T}+1.
\end{equation*}
We only discuss the case when $T>\rho$ (i.e., $M_{T}>M$), otherwise, the result is immediately obtained.
For $0\leq k\leq M-1$, \eqref{5.20} leads to
\begin{equation}\label{5.21}
\mathbb{E}\left(\sup_{t_{k}\leq t\leq t_{k+1}}|\bar{X}^{i,N}(t)-X^{i,N}(t)|^{\hat{p}_{1}}\right)\leq C\Delta^{\hat{p}_{1}/2}.
\end{equation}
For $M\leq k\leq 2M-1$, combining \eqref{5.20} and \eqref{5.21} with the H\"{o}lder inequality gives that
\begin{equation*}
\begin{split}
&\mathbb{E}\left(\sup_{t_{k}\leq t\leq t_{k+1}}|\bar{X}^{i,N}(t)-X^{i,N}(t)|^{\hat{p}_{2}}\right)
\\&\leq C\Delta^{\hat{p}_{2}/2}+C\left[\mathbb{E}\left(\sup_{t_{k}\leq t\leq t_{k+1}}|\bar{X}^{i,N}(t-\rho)-X^{i,N}(t-\rho)|^{\hat{p}_{1}}\right)\right]^{\hat{p}_{2}/\hat{p}_{1}}
\\&\leq C\Delta^{\hat{p}_{2}/2}.
\end{split}
\end{equation*}
The desired result is established by induction.
\end{proof}

\begin{thm}\label{thm5.3}
Let Assumptions \ref{Halpha}-\ref{HD} hold. Then, for any $p\in \left[2, \bar{p}\left(\frac{\hat{p}}{\bar{p}+\hat{p}l_{1}}\wedge\frac{1}{2(l_{U}+1)}\right)\right]$ and $4l_{U}+4\leq (2l_{U}+2)\hat{p}\leq\bar{p}$, we have
\begin{equation}\label{5.51}
\max_{i\in \mathbb{N}_N}\mathbb{E}\left(\sup_{0\leq t\leq T}|Y^{i,N}(t)-\bar{X}^{i,N}(t)|^{p}\right)
\leq C\Delta^{\gamma p}, \quad \forall T>0,
\end{equation}
and
\begin{equation}\label{5.52}
\max_{i\in \mathbb{N}_N}\mathbb{E}\left(\sup_{0\leq t\leq T}|Y^{i,N}(t)-X^{i,N}(t)|^{p}\right)
\leq C\Delta^{\gamma p}, \quad \forall T>0.
\end{equation}
\end{thm}
\begin{proof}
For any $i\in \mathbb{S}_N$ and $t\in [0,T]$, set
$$\Theta^{i,N}(t)=Y^{i,N}(t)-\bar{X}^{i,N}(t)-D(Y^{i,N}(t-\rho))+D(\bar{X}^{i,N}(t-\rho)).$$
Application of It\^{o}’s formula yields that
\begin{equation*}
\begin{split}
&|\Theta^{i,N}(t)|^{p}\\&\leq p\int_{0}^{t}|\Theta^{i,N}(s)|^{p-2}(\Theta^{i,N}(s))^{T}
\left[\alpha\left(\Gamma(Y^{i,N}_{s}),\mathbb{L}_{\Gamma(Y^{N}_{s})}\right)
-\alpha_{\Delta}\left(\Gamma(X^{i,N}_{s}),\mathbb{L}_{\Gamma(X^{N}_{s})}\right)\right]ds
\\&+\frac{p(p-1)}{2}\int_{0}^{t}|\Theta^{i,N}(s)|^{p-2}\left|\beta\left(\Gamma(Y^{i,N}_{s}),\mathbb{L}_{\Gamma(Y^{N}_{s})}\right)
-\beta\left(\Gamma(X^{i,N}_{s}),\mathbb{L}_{\Gamma(X^{N}_{s})}\right)\right|^{2}ds
\\&+p\int_{0}^{t}|\Theta^{i,N}(s)|^{p-2}(\Theta^{i,N}(s))^{T}
\\&\quad \quad \quad\cdot\left[\beta\left(\Gamma(Y^{i,N}_{s}),\mathbb{L}_{\Gamma(Y^{N}_{s})}\right)
-\beta\left(\Gamma(X^{i,N}_{s}),\mathbb{L}_{\Gamma(X^{N}_{s})}\right)\right]dB^{i}(s)
\\&\leq p\int_{0}^{t}|\Theta^{i,N}(s)|^{p-2}(\Theta^{i,N}(s))^{T}
\left[\alpha\left(\Gamma(Y^{i,N}_{s}),\mathbb{L}_{\Gamma(Y^{N}_{s})}\right)
-\alpha\left(\Gamma(\bar{X}^{i,N}_{s}),\mathbb{L}_{\Gamma(\bar{X}^{N}_{s})}\right)\right]ds
\\&+ p\int_{0}^{t}|\Theta^{i,N}(s)|^{p-2}(\Theta^{i,N}(s))^{T}
\left[\alpha\left(\Gamma(\bar{X}^{i,N}_{s}),\mathbb{L}_{\Gamma(\bar{X}^{N}_{s})}\right)
-\alpha\left(\Gamma(\bar{X}^{i,N}_{s}),\mathbb{L}_{\Gamma(X^{N}_{s})}\right)\right]ds
\\&+p\int_{0}^{t}|\Theta^{i,N}(s)|^{p-2}(\Theta^{i,N}(s))^{T}
\left[\alpha\left(\Gamma(\bar{X}^{i,N}_{s}),\mathbb{L}_{\Gamma(X^{N}_{s})}\right)
-\alpha\left(\Gamma(X^{i,N}_{s}),\mathbb{L}_{\Gamma(X^{N}_{s})}\right)\right]ds
\\&+p\int_{0}^{t}|\Theta^{i,N}(s)|^{p-2}(\Theta^{i,N}(s))^{T}
\cdot\left[\alpha\left(\Gamma(X^{i,N}_{s}),\mathbb{L}_{\Gamma(X^{N}_{s})}\right)
-\alpha_{\Delta}\left(\Gamma(X^{i,N}_{s}),\mathbb{L}_{\Gamma(X^{N}_{s})}\right)\right]ds
\\&+\frac{p(p-1)}{2}\int_{0}^{t}|\Theta^{i,N}(s)|^{p-2}\left|\beta\left(\Gamma(Y^{i,N}_{s}),\mathbb{L}_{\Gamma(Y^{N}_{s})}\right)
-\beta\left(\Gamma(X^{i,N}_{s}),\mathbb{L}_{\Gamma(X^{N}_{s})}\right)\right|^{2}ds
\\&+p\int_{0}^{t}|\Theta^{i,N}(s)|^{p-2}(\Theta^{i,N}(s))^{T}
\\&\quad \quad \quad\cdot\left[\beta\left(\Gamma(Y^{i,N}_{s}),\mathbb{L}_{\Gamma(Y^{N}_{s})}\right)
-\beta\left(\Gamma(X^{i,N}_{s}),\mathbb{L}_{\Gamma(X^{N}_{s})}\right)\right]dB^{i}(s)
\\&=:A^{i,N}_{1}(t)+A^{i,N}_{2}(t)+A^{i,N}_{3}(t)+A^{i,N}_{4}(t)+A^{i,N}_{5}(t)+A^{i,N}_{6}(t).
\end{split}
\end{equation*}
Using  Lemma \ref{lem5.1}, Young's inequality and Assumption \ref{Halpha}, we have
\begin{equation*}
\begin{split}
&\mathbb{E}\left(\sup_{0\leq s\leq t}A^{i,N}_{1}(s)\right)
\\&\leq C\mathbb{E}\int_{0}^{t}\left(\left|Y^{i,N}(s)-\bar{X}^{i,N}(s)\right|^{p-2}
 +\left|D(Y^{i,N}(s-\rho))-D(\bar{X}^{i,N}(s-\rho))\right|^{p-2}\right)
\\&\quad \quad \quad \cdot\left[\sum_{v=1}^{r-1}|Y^{i,N}(s+\bar{s}_{v})-\bar{X}^{i,N}(s+\bar{s}_{v})|^{2}\right.
\\&\quad\quad \quad +U_{2}^{2}(Y^{i,N}(s-\rho),\bar{X}^{i,N}(s-\rho))|Y^{i,N}(s-\rho)-\bar{X}^{i,N}(s-\rho)|^{2}
\\&\quad \quad\quad\left. +\sum_{v=1}^{r}\mathbb{W}_{2}^{2}(\mathbb{L}_{Y^{N}(s+\bar{s}_{v})},\mathbb{L}_{\bar{X}^{N}(s+\bar{s}_{v})})\right]ds
\\
\end{split}
\end{equation*}
\begin{equation*}
\begin{split}
&\leq C\mathbb{E}\int_{0}^{t}\left[\sum_{v=1}^{r-1}|Y^{i,N}(s+\bar{s}_{v})-\bar{X}^{i,N}(s+\bar{s}_{v})|^{p}\right.
\\&\quad +\left[U_{2}(Y^{i,N}(s-\rho),\bar{X}^{i,N}(s-\rho))\vee U_{3}(Y^{i,N}(s-\rho),\bar{X}^{i,N}(s-\rho))\right]^{p}
\\&\quad \quad \quad\cdot|Y^{i,N}(s-\rho)-\bar{X}^{i,N}(s-\rho)|^{p}
\left. +\sum_{v=1}^{r}\mathbb{W}_{2}^{2}(\mathbb{L}_{Y^{N}(s+\bar{s}_{v})},\mathbb{L}_{\bar{X}^{N}(s+\bar{s}_{v})})\right]ds
\\&\leq C\int_{0}^{t}\mathbb{E}\left(\sup_{0\leq u\leq s}|Y^{i,N}(u)-\bar{X}^{i,N}(u)|^{p}\right)ds
\\&\quad +C\int_{0}^{t}\left(\mathbb{E}\left|Y^{i,N}(s-\rho)-\bar{X}^{i,N}(s-\rho)\right|^{2p}\right)^{1/2}ds
\\&\quad +C\mathbb{E}\int_{0}^{t}\sum_{v=1}^{r}\frac{1}{N}\sum_{j=1}^{N}\left|Y^{j,N}(s+\bar{s}_{v})
-\bar{X}^{j,N}(s+\bar{s}_{v})\right|^{p}ds
\\&\leq C\int_{0}^{t}\mathbb{E}\left(\sup_{0\leq u\leq s}|Y^{i,N}(u)-\bar{X}^{i,N}(u)|^{p}\right)ds
\\&\quad +C\int_{0}^{t}\left(\mathbb{E}\left|Y^{i,N}(s-\rho)-\bar{X}^{i,N}(s-\rho)\right|^{2p}\right)^{1/2}ds.
\end{split}
\end{equation*}
Applying Lemma \ref{lem5.2} and Assumption \ref{Halpha} gives that
\begin{equation*}
	\begin{split}
&\mathbb{E}\left(\sup_{0\leq s\leq t}A^{i,N}_{2}(s)\right)
\\&\leq C\mathbb{E}\int_{0}^{t}\left|\Theta^{i,N}(s)\right|^{p-1}
\sum_{v=1}^{r}\mathbb{W}_{2}(\mathbb{L}_{\bar{X}^{N}(s+\bar{s}_{v})},\mathbb{L}_{X^{N}(s+\bar{s}_{v})})ds
\\&\leq C\mathbb{E}\left[\sup_{0\leq s\leq t}\left|\Theta^{i,N}(s)\right|^{p-1}\int_{0}^{t}
\sum_{v=1}^{r}\mathbb{W}_{2}(\mathbb{L}_{\bar{X}^{N}(s+\bar{s}_{v})},\mathbb{L}_{X^{N}(s+\bar{s}_{v})})ds\right]
\\&\leq \frac{1}{5}\mathbb{E}\left(\sup_{0\leq s\leq t}\left|\Theta^{i,N}(s)\right|^{p}\right)
+C\mathbb{E}\int_{0}^{t}
\sum_{v=1}^{r}\frac{1}{N}\sum_{j=1}^{N}\left|\bar{X}^{j,N}(s+\bar{s}_{v})
-X^{j,N}(s+\bar{s}_{v})\right|^{p}ds
\\&\leq \frac{1}{5}\mathbb{E}\left(\sup_{0\leq s\leq t}\left|\Theta^{i,N}(s)\right|^{p}\right)+C\Delta^{\frac{p}{2}}.
\end{split}
\end{equation*}
By Lemma \ref{lem5.1}, H\"{o}lder's inequality, Young's inequality and Assumptions \ref{Halpha}, \ref{initial}, we derive that
\begin{equation*}
\begin{split}
&\mathbb{E}\left(\sup_{0\leq s\leq t}A^{i,N}_{3}(s)\right)
\\&\leq C\mathbb{E}\int_{0}^{t}\left|\Theta^{i,N}(s)\right|^{p-1}
\\&\quad \quad \quad\cdot\sum_{v=1}^{r}\left[U_{1}(\bar{X}^{i,N}(s+\bar{s}_{v}),X^{i,N}(s+\bar{s}_{v}))
|\bar{X}^{i,N}(s+\bar{s}_{v})-X^{i,N}(s+\bar{s}_{v})|\right]ds
\\
\end{split}
\end{equation*}
\begin{equation*}
\begin{split}
&\leq C\mathbb{E}\left[\sup_{0\leq s\leq t}\left|\Theta^{i,N}(s)\right|^{p-1}\right.
\\&\quad \quad \quad\left.\cdot\int_{0}^{t}\sum_{v=1}^{r}\left[U_{1}(\bar{X}^{i,N}(s+\bar{s}_{v}),X^{i,N}(s+\bar{s}_{v}))
|\bar{X}^{i,N}(s+\bar{s}_{v})-X^{i,N}(s+\bar{s}_{v})|\right]ds\right]
\\&\leq \frac{1}{5}\mathbb{E}\left(\sup_{0\leq s\leq t}\left|\Theta^{i,N}(s)\right|^{p}\right)
\\&\quad +C\mathbb{E}\left(\int_{0}^{t}\sum_{v=1}^{r}\left[U_{1}(\bar{X}^{i,N}(s+\bar{s}_{v}),X^{i,N}(s+\bar{s}_{v}))
|\bar{X}^{i,N}(s+\bar{s}_{v})-X^{i,N}(s+\bar{s}_{v})|\right]ds\right)^{p}
\\&\leq \frac{1}{5}\mathbb{E}\left(\sup_{0\leq s\leq t}\left|\Theta^{i,N}(s)\right|^{p}\right)
\\&\quad +C\int_0^t\sum_{v=1}^{r}\left[\mathbb{E}\left(1
+|\bar{X}^{i,N}(s+\bar{s}_{v})|^{pl_1}+|X^{i,N}(s+\bar{s}_{v})|^{pl_1}\right)^{\frac{\bar{p}}{pl_1}}\right]^{\frac{pl_1}{\bar{p}}}
\\&\quad \quad \quad \quad\left[\mathbb{E}\left(|\bar{X}^{i,N}(s+\bar{s}_{v})-X^{i,N}(s+\bar{s}_{v})|^{p}\right)^{\frac{\bar{p}}{\bar{p}-pl_1}}\right]^{\frac{\bar{p}-pl_1}{\bar{p}}}ds
\\&\leq \frac{1}{5}\mathbb{E}\left(\sup_{0\leq s\leq t}\left|\Theta^{i,N}(s)\right|^{p}\right)
\\&\quad +C\int_{0}^{t}
\left[\mathbb{E}\left(\sup_{0\leq u\leq s}(|\bar{X}^{i,N}(u)-X^{i,N}(u)|^{\frac{p\bar{p}}{\bar{p}-pl_1}})\right)\right]^{\frac{\bar{p}-pl_1}{\bar{p}}}ds
\\&\leq \frac{1}{5}\mathbb{E}\left(\sup_{0\leq s\leq t}\left|\Theta^{i,N}(s)\right|^{p}\right)+C\Delta^{\frac{p}{2}}.
\end{split}
\end{equation*}
By Assumptions \ref{Halpha}, H\"{o}lder's inequality, Young's inequality and (\ref{5.1}), we get that
\begin{equation*}
\begin{split}
&\mathbb{E}\left(\sup_{0\leq s\leq t}A^{i,N}_{4}(s)\right)
\\&\leq C\mathbb{E}\int_{0}^{t}\left|\Theta^{i,N}(s)\right|^{p-1}
\left|\alpha\left(\Gamma(X^{i,N}_{s}),\mathbb{L}_{\Gamma(X^{N}_{s})}\right)
-\alpha_{\Delta}\left(\Gamma(X^{i,N}_{s}),\mathbb{L}_{\Gamma(X^{N}_{s})}\right)\right|ds
\\&\leq C\mathbb{E}\left(\sup_{0\leq s\leq t}\left|\Theta^{i,N}(s)\right|^{p-1}\right.
\left.\int_{0}^{t}\left|\alpha\left(\Gamma(X^{i,N}_{s}),\mathbb{L}_{\Gamma(X^{N}_{s})}\right)
-\alpha_{\Delta}\left(\Gamma(X^{i,N}_{s}),\mathbb{L}_{\Gamma(X^{N}_{s})}\right)\right|ds\right)
\\&\leq \frac{1}{5}\mathbb{E}\left(\sup_{0\leq s\leq t}\left|\Theta^{i,N}(s)\right|^{p}\right)
\\&\quad +C\mathbb{E}\int_{0}^{t}\left|\alpha\left(\Gamma(X^{i,N}_{s}),\mathbb{L}_{\Gamma(X^{N}_{s})}\right)
-\alpha_{\Delta}\left(\Gamma(X^{i,N}_{s}),\mathbb{L}_{\Gamma(X^{N}_{s})}\right)\right|^{p}ds
\\&\leq \frac{1}{5}\mathbb{E}\left(\sup_{0\leq s\leq t}\left|\Theta^{i,N}(s)\right|^{p}\right)
 +C\Delta^{p\gamma}\int_{0}^{t}\mathbb{E}\left[\frac{\left|\alpha\left(\Gamma(X^{i,N}_{s}),\mathbb{L}_{\Gamma(X^{N}_{s})}\right)\right|^{2p}}
{\left(1+\Delta^{\gamma}|\alpha(\Gamma(X^{i,N}_{s}),\mathbb{L}_{\Gamma(X^{N}_{s})})|\right)^{p}}\right]ds
\\
\end{split}
\end{equation*}
\begin{equation*}
\begin{split}
&\leq \frac{1}{5}\mathbb{E}\left(\sup_{0\leq s\leq t}\left|\Theta^{i,N}(s)\right|^{p}\right)
 +C\Delta^{p\gamma}\int_{0}^{t}\mathbb{E}\left(1+\sum_{v=1}^{r}\left[\mathcal{W}_{2}(\mathbb{L}_{X^{N}(s+\bar{s}_{v})})\right.\right.
\\&\quad \left.\left.+U_{1}(X^{i,N}(s+\bar{s}_{v}),0)
|X^{i,N}(s+\bar{s}_{v})|\right]\right)^{2p}ds
\\&\leq \frac{1}{5}\mathbb{E}\left(\sup_{0\leq s\leq t}\left|\Theta^{i,N}(s)\right|^{p}\right)+C\Delta^{p\gamma}.
\end{split}
\end{equation*}
Applying Young's inequality, H\"{o}lder's inequality, Assumptions \ref{Hbeta}-\ref{initial} and the estimation of $A^{i,N}_{1}(s)$  gives that
\begin{equation*}
\begin{split}
&\mathbb{E}\left(\sup_{0\leq s\leq t}A^{i,N}_{5}(s)\right)
\\&\leq C\mathbb{E}\int_{0}^{t}\left(\left|Y^{i,N}(s)-\bar{X}^{i,N}(s)\right|^{p-2}+\left|D(Y^{i,N}(s-\rho))-D(\bar{X}^{i,N}(s-\rho))\right|^{p-2}\right)
\\&\quad \quad \quad\left[\sum_{v=1}^{r-1}|Y^{i,N}(s+\bar{s}_{v})-X^{i,N}(s+\bar{s}_{v})|^{2}\right.
\\&\quad +U_{2}^{2}(Y^{i,N}(s-\rho),X^{i,N}(s-\rho))|Y^{i,N}(s-\rho)-X^{i,N}(s-\rho)|^{2}
\\&\left.\quad +\sum_{v=1}^{r}\mathbb{W}_{2}^{2}(\mathbb{L}_{Y^{N}(s+\bar{s}_{v})},\mathbb{L}_{X^{N}(s+\bar{s}_{v})})\right]ds
\\&\leq C\mathbb{E}\int_{0}^{t}\left[\left|Y^{i,N}(s)-\bar{X}^{i,N}(s)\right|^{p}\right.
\\&\quad +U^{p}_{3}(Y^{i,N}(s-\rho),\bar{X}^{i,N}(s-\rho))|Y^{i,N}(s-\rho)-\bar{X}^{i,N}(s-\rho)|^{p}
\\&\quad \left.+\sum_{v=1}^{r-1}|Y^{i,N}(s+\bar{s}_{v})-X^{i,N}(s+\bar{s}_{v})|^{p}\right.
\\&\quad +U_{2}^{p}(Y^{i,N}(s-\rho),X^{i,N}(s-\rho))|Y^{i,N}(s-\rho)-X^{i,N}(s-\rho)|^{p}
\\&\left.\quad +\sum_{v=1}^{r}\mathbb{W}_{2}^{p}(\mathbb{L}_{Y^{N}(s+\bar{s}_{v})},\mathbb{L}_{X^{N}(s+\bar{s}_{v})})\right]ds
\\&\leq C\int_{0}^{t}\mathbb{E}\left(\sup_{0\leq u\leq s}\left|Y^{i,N}(u)-\bar{X}^{i,N}(u)\right|^{p}\right)ds
\\&\quad +C\int_{0}^{t}\left(\mathbb{E}|Y^{i,N}(s-\rho)-\bar{X}^{i,N}(s-\rho)|^{2p}\right)^{1/2}ds
\\&\quad +C\int_{0}^{t}\mathbb{E}\left(\sum_{v=1}^{r}\left|\bar{X}^{i,N}(s+\bar{s}_{v})-X^{i,N}(s+\bar{s}_{v})\right|^{p}\right)ds
\\&\quad +C\int_{0}^{t}\mathbb{E}\left[\sum_{v=1}^{r}\left(\frac{1}{N}\sum_{j=1}^{N}\left|Y^{j,N}(s+\bar{s}_{v})
-X^{j,N}(s+\bar{s}_{v})\right|^{p}\right)\right]ds
\\
\end{split}
\end{equation*}
\begin{equation*}
\begin{split}
&\leq C\int_{0}^{t}\mathbb{E}\left(\sup_{0\leq u\leq s}\left|Y^{i,N}(u)-\bar{X}^{i,N}(u)\right|^{p}\right)ds
\\&\quad +C\int_{0}^{t}\left(\mathbb{E}|Y^{i,N}(s-\rho)-\bar{X}^{i,N}(s-\rho)|^{2p}\right)^{1/2}ds+C\Delta^{\frac{p}{2}}.
\end{split}
\end{equation*}
Similar to the estimation of $A^{i,N}_{5}(t)$, using BDG's inequality, Young's inequality means that
\begin{equation*}
\begin{split}
&\mathbb{E}\left(\sup_{0\leq s\leq t}A^{i,N}_{6}(s)\right)
\\&\leq C\mathbb{E}\left[\int_{0}^{t}\left|\Theta^{i,N}(s)\right|^{2p-2}\left|\beta\left(\Gamma(Y^{i,N}_{s}),\mathbb{L}_{\Gamma(Y^{N}_{s})}\right)
-\beta\left(\Gamma(X^{i,N}_{s}),\mathbb{L}_{\Gamma(X^{N}_{s})}\right)\right|^{2}ds\right]^{1/2}
\\&\leq \frac{1}{5}\mathbb{E}\left(\sup_{0\leq s\leq t}\left|\Theta^{i,N}(s)\right|^{p}\right)
\\&\quad +C\mathbb{E}\left[\int_{0}^{t}\left|\beta\left(\Gamma(Y^{i,N}_{s}),\mathbb{L}_{\Gamma(Y^{N}_{s})}\right)
-\beta\left(\Gamma(X^{i,N}_{s}),\mathbb{L}_{\Gamma(X^{N}_{s})}\right)\right|^{2}ds\right]^{p/2}
\\&\leq \frac{1}{5}\mathbb{E}\left(\sup_{0\leq s\leq t}\left|\Theta^{i,N}(s)\right|^{p}\right)
 +C\mathbb{E}\int_{0}^{t}\left[\sum_{v=1}^{r-1}|Y^{i,N}(s+\bar{s}_{v})-X^{i,N}(s+\bar{s}_{v})|^{p}\right.
\\&\quad +U_{2}^{p}(Y^{i,N}(s-\rho),X^{i,N}(s-\rho))|Y^{i,N}(s-\rho)-X^{i,N}(s-\rho)|^{p}
\\&\left.\quad +\sum_{v=1}^{r}\mathbb{W}_{2}^{p}(\mathbb{L}_{Y^{N}(s+\bar{s}_{v})},\mathbb{L}_{X^{N}(s+\bar{s}_{v})})\right]ds
\\&\leq \frac{1}{5}\mathbb{E}\left(\sup_{0\leq s\leq t}\left|\Theta^{i,N}(s)\right|^{p}\right)+C\Delta^{\frac{p}{2}}.
\end{split}
\end{equation*}
Combining these inequalities gives that
\begin{equation*}
\begin{split}
&\mathbb{E}\left(\sup_{0\leq s\leq t}\left|\Theta^{i,N}(s)\right|^{p}\right)
\\&\leq \frac{4}{5}\mathbb{E}\left(\sup_{0\leq s\leq t}\left|\Theta^{i,N}(s)\right|^{p}\right)
+C\int_{0}^{t}\mathbb{E}\left(\sup_{0\leq u\leq s}\left|Y^{i,N}(u)-\bar{X}^{i,N}(u)\right|^{p}\right)ds
\\&\quad +C\int_{0}^{t}\left(\mathbb{E}|Y^{i,N}(s-\rho)-\bar{X}^{i,N}(s-\rho)|^{2p}\right)^{1/2}ds+C\Delta^{p\gamma}.
\end{split}
\end{equation*}
Therefore, using Assumption \ref{HD} yields that
\begin{equation*}
\begin{split}
&\mathbb{E}\left(\sup_{0\leq s\leq t}\left|Y^{i,N}(s)-\bar{X}^{i,N}(s)\right|^{p}\right)
\\&\leq C\mathbb{E}\left(\sup_{0\leq s\leq t}\left|\Theta^{i,N}(s)\right|^{p}\right)
\\&\quad +C\mathbb{E}\left(\sup_{0\leq s\leq t}\left|D(Y^{i,N}(s-\rho))-D(\bar{X}^{i,N}(s-\rho))\right|^{p}\right)
\\&\leq C\int_{0}^{t}\mathbb{E}\left(\sup_{0\leq u\leq s}\left|Y^{i,N}(u)-\bar{X}^{i,N}(u)\right|^{p}\right)ds
\\&\quad +C\left[\mathbb{E}\left(\sup_{0\leq s\leq t}|Y^{i,N}(s-\rho)-\bar{X}^{i,N}(s-\rho)|^{2p}\right)\right]^{1/2}+C\Delta^{p\gamma}.
\end{split}
\end{equation*}
Thanks to Gronwall's inequality, we have
\begin{equation*}
\begin{split}
&\mathbb{E}\left(\sup_{0\leq s\leq t}\left|Y^{i,N}(s)-\bar{X}^{i,N}(s)\right|^{p}\right)
\\&\leq C\Delta^{p\gamma}+C\left[\mathbb{E}\left(\sup_{0\leq s\leq t}|Y^{i,N}(s-\rho)-\bar{X}^{i,N}(s-\rho)|^{2p}\right)\right]^{1/2}.
\end{split}
\end{equation*}
Define $p_{j}=(M_T+2-j)p2^{M_T+1-j}$, $j=1,2,\ldots,M_{T}+1$. Applying the same technique in Lemma \ref{lem5.2} gives the desired result \eqref{5.51}. Then \eqref{5.52} is obtained by Lemma \ref{lem5.2} and \eqref{5.51} immediately. 
\end{proof}

\begin{rem}
The  optimal convergence rate is $\frac{1}{2}$, which can be obtained in Theorem \ref{thm5.3} by choosing $\gamma=\frac{1}{2}$.
\end{rem}

\begin{thm}\label{thm5.4}
Let the assumptions in Theorems \ref{thm4.1} and \ref{thm5.3} hold. Then, 
\begin{equation*}
\begin{split}
\mathbb{E}&\left(\sup_{0\leq t\leq T}|Y^{i}(t)-\bar{X}^{i,N}(t)|^{p}\right)
\\&\leq C\left\{\begin{array}{ll}
(N^{-1 / 2})^{\lambda_{T,\rho,p}}+\Delta^{p\gamma}, & \text { if } p>d/2, \\
{[N^{-1 / 2} \log(1+N)]^{\lambda_{T,\rho,p}}}+\Delta^{p\gamma}, & \text { if } p=d/2,\\
{(N^{-p / d})^{\lambda_{T,\rho,p}}}+\Delta^{p\gamma}, & \text { if }2\leq p<d/2,
\end{array}\right.
\end{split}
\end{equation*}
where $\lambda_{T,\rho,p}=\left(\frac{p-\varepsilon}{p}\right)^{\lfloor\frac{T}{\rho}\rfloor}$.
\end{thm}
The above theorem follows directly from Theorems \ref{thm4.1} and \ref{thm5.3}.

\begin{thm}\label{ththmm}
Let the assumptions in Theorems \ref{thm4.2} and \ref{thm5.3} hold. Then, 
\begin{equation*}
\begin{split}
\mathbb{E}&\left(\sup_{0\leq t\leq T}|Y^{i}(t)-\bar{X}^{i,N}(t)|^{p}\right)
\\&\leq C\left\{\begin{array}{ll}
N^{-1 / 2}+N^{-(\bar{p}-p)/\bar{p}}+\Delta^{p\gamma}, & \text { if } p>d/2\text{ and }\bar{p}\neq2p, \\
N^{-1 / 2} \log(1+N)+N^{-(\bar{p}-p)/\bar{p}}+\Delta^{p\gamma}, & \text { if } p=d/2\text{ and }\bar{p}\neq2p,\\
N^{-p / d}+N^{-(\bar{p}-p)/\bar{p}}+\Delta^{p\gamma}, & \text { if }2\leq p<d/2.
\end{array}\right.
\end{split}
\end{equation*}
\end{thm}
The above theorem can be achieved easily by using Theorems \ref{thm4.2} and \ref{thm5.3}.

\section{Example}
We would like to point out that our theory can cover the numerical example (5.1) in \cite{21}. Moreover, the delay component in our results can be more general. In the following two examples, we can see that the neutral term is highly nonlinear and the constraint of delay variables is general.
\begin{exa}\label{ex1}
	Consider the scalar NMSMVE
	\begin{equation}\label{exam1}
	\begin{split}
		&d[Y(t)+Y^3(t-2)]\\
		&=[-2Y(t)+Y(t-0.25)-2Y^5(t-2)+\mathbb{E}Y(t-0.25)-\mathbb{E}Y(t-0.5)]dt\\
		&~~~+[Y(t)+0.25Y(t-0.2)+\mathbb{E}Y(t-0.4)]dB(t),
	\end{split}
\end{equation}
	with the initial data $\xi(t)=|t|^\frac{1}{2}+4$, $t\in[-2,0]$. 
	All assumptions are satisfied. We only check the third inequality in Assumption \ref{Halpha}:
\begin{equation*}
	\begin{split}
		&(x_1+x_5^3-y_1-y_5^3)(-2x_1+x_3-2x_5^5+2y_1-y_3+2y_5^5)\\
		&=(x_1-y_1)(-2x_1+2y_1+x_3-y_3)+(x_1-y_1)(-2x_5^5+2y_5^5)\\
		&+(x_5^3-y_5^3)(-2x_1+2y_1+x_3-y_3)+(x_5^3-y_5^3)(-2x_5^5+2y_5^5)\\
		&\leq -2|x_1-y_1|^2+\frac{1}{2}|x_1-y_1|^2+\frac{1}{2}|x_3-y_3|^2+\frac{1}{2}|x_1-y_1|^2+\\
		&+\frac{1}{2}|x_5-y_5|^2|x_5^4+x_5^3y_5+x_5^2y_5^2+x_5y_5^3+y_5^4|^2\\
		&+\frac{1}{2}|x_5-y_5|^2|x_5^2+x_5y_5+y_5^2|^2+4|x_1-y_1|^2+|x_3-y_3|^2\\
		&+|x_5-y_5|^2(x_5^2+x_5y_5+y_5^2)(x_5^4+x_5^3y_5+x_5^2y_5^2+x_5y_5^3+y_5^4)\\
		&\leq 3|x_1-y_1|^2+\frac{3}{2}|x_3-y_3|^2+25(1+x_5^8+y_5^8)|x_5-y_5|^2.
	\end{split}
\end{equation*}
Hence, Assumption \ref{Halpha} holds with $U_2(x_5,y_5)=1+x_5^4+y_5^4$. 
\end{exa}
\begin{figure}[htbp] 
	\centering
	\includegraphics[height=6.5cm,width=8cm]{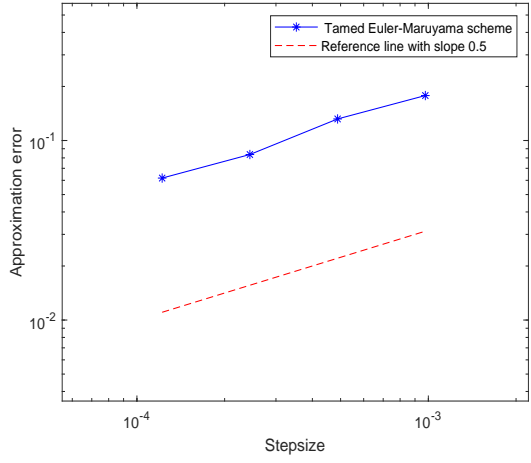}
	\caption{\label{tu1} The convergence rate of tamed EM scheme for (\ref{exam1})}
\end{figure}
\begin{exa}\label{ex2}
	Consider the two-dimensional NMSMVE
	\begin{equation}\label{exam2}
		\begin{split}
			&d\left[\begin{array}{c}
				Y_1(t)+2\sin (Y_1(t-4)) \\
				Y_2(t)+4Y_2^2(t-4)
			\end{array}\right]\\
			&=\left[\begin{array}{c}
			-3Y_2(t)+Y_1(t-0.1)-4Y_1^3(t-4)+\mathbb{E}Y_2(t-0.4)-3\mathbb{E}Y_1(t-1) \\
			-4(Y_2(t)+4Y_2^2(t-4))|Y_2(t)+4Y_2^2(t-4)|+30Y_2^2(t-4)+6Y_2(t)-2Y_2^5(t-4)
		\end{array}\right]dt\\
			&~~~+\left[\begin{array}{c}
				2Y_2(t)+\mathbb{E}Y_1(t-0.1) \\
				4Y_1(t)+\mathbb{E}Y_2(t-0.1)
			\end{array}\right]dB(t),
		\end{split}
	\end{equation}
	with the initial data $\xi(t)=|t|^\frac{2}{3}+1$, $t\in[-4,0]$. Let us make an explanation for $x_{ab}\in\mathbb{R}$: it is the $b$-th  value of the $a$-th element in the solution $Y(t)$. For example, $Y_1(t)=x_{11}, Y_2(t)=x_{21}, Y_1(t-4)=x_{14}, Y_2(t-4)=x_{24}$.
	All assumptions are satisfied. We only check the third inequality in Assumption \ref{Halpha}:
	\begin{equation*}
		\begin{split}
			&\left[\left(\begin{array}{c}
				x_{11}+2\sin(x_{14}) \\
				x_{21}+4x_{24}^2
			\end{array}\right)-\left(\begin{array}{c}
			y_{11}+2\sin(y_{14}) \\
			y_{21}+4y_{24}^2
		\end{array}\right)\right]^T\\
			&\left[\left(\begin{array}{c}
				-3x_{21}+x_{12}-4x_{14}^3\\
				-4(x_{21}+4x_{24}^2)|x_{21}+4x_{24}^2|+30x_{24}^2+6x_{21}-2x_{24}^5
			\end{array}\right)\right.\\
		&~~~\left.-\left(\begin{array}{c}
				-3y_{21}+y_{12}-4y_{14}^3\\
				-4(y_{21}+4y_{24}^2)|y_{21}+4y_{24}^2|+30y_{24}^2+6y_{21}-2y_{24}^5
			\end{array}\right)\right]\\
			&=[x_{11}+2\sin(x_{14})-y_{11}-2\sin(y_{14})][-3x_{21}+x_{12}-4x_{14}^3+3y_{21}-y_{12}+4y_{14}^3]\\
			&~~~+[x_{21}+4x_{24}^2-y_{21}-4y_{24}^2][-4(x_{21}+4x_{24}^2)|x_{21}+4x_{24}^2|+30x_{24}^2+6x_{21}-2x_{24}^5\\
			&~~~~~~+4(y_{21}+4y_{24}^2)|y_{21}+4y_{24}^2|-30y_{24}^2-6y_{21}+2y_{24}^5]\\
			&=:NU_1+NU_2.
			\end{split}
	\end{equation*}
Using the Young inequality and the inequality $|\sin A-\sin B|\leq|A-B|$ for any $A,B\in \mathbb{R}$ gives that
		\begin{equation*}
		\begin{split}	
			NU_1=&-3(x_{11}-y_{11})(x_{21}-y_{21})+(x_{11}-y_{11})(x_{12}-y_{12})\\
			&-4(x_{11}-y_{11})(x_{14}^3-y_{14}^3-6(\sin(x_{14})-\sin(y_{14}))(x_{21}-y_{21}))\\
			&+2(\sin(x_{14})-\sin(y_{14}))(x_{12}-y_{12})-8(\sin(x_{14})-\sin(y_{14}))(x_{14}^3-y_{14}^3)\\
			\leq& \frac{3}{2}|x_{11}-y_{11}|^2+\frac{3}{2}|x_{21}-y_{21}|^2+\frac{1}{2}|x_{11}-y_{11}|^2+\frac{1}{2}|x_{12}-y_{12}|^2\\
			&+2|x_{11}-y_{11}|^2+2|x_{14}-y_{14}|^2|x_{14}^2+x_{14}y_{14}+y_{14}^2|^2\\
			&+3|x_{14}-y_{14}|^2+3|x_{21}-y_{21}|^2+|x_{14}-y_{14}|^2+|x_{12}-y_{12}|^2\\
			&+4|x_{14}-y_{14}|^2+4|x_{14}-y_{14}|^2|x_{14}^2+x_{14}y_{14}+y_{14}^2|^2.
		\end{split}
	\end{equation*}
Note that
\begin{equation*}
	\begin{split}	
		NU_2=&[(x_{21}+4x_{24}^2)-(y_{21}+4y_{24}^2)][-4(x_{21}+4x_{24}^2)|x_{21}+4x_{24}^2|\\
		&~~~+4(y_{21}+4y_{24}^2)|y_{21}+4y_{24}^2|+6(x_{21}+4x_{24}^2)-6(y_{21}+4y_{24}^2)]\\
		&+6[(x_{21}+4x_{24}^2)-(y_{21}+4y_{24}^2)](x_{24}^2-y_{24}^2)\\
		&-2[(x_{21}+4x_{24}^2)-(y_{21}+4y_{24}^2)](x_{24}^5-y_{24}^5).
			\end{split}
	\end{equation*}
Applying the inequality $(|A|+|B|)(|A|-|B|)^2\leq(A-B)(A|A|-B|B|)$ for any $A,B\in \mathbb{R}$ leads to
	\begin{equation*}
		\begin{split}
			&[(x_{21}+4x_{24}^2)-(y_{21}+4y_{24}^2)][-4(x_{21}+4x_{24}^2)|x_{21}+4x_{24}^2|\\
			&~~~+4(y_{21}+4y_{24}^2)|y_{21}+4y_{24}^2|+6(x_{21}+4x_{24}^2)-6(y_{21}+4y_{24}^2)]\\	
		&\leq -4[|x_{21}+4x_{24}^2|+|y_{21}+4y_{24}^2|][|x_{21}+4x_{24}^2|-|y_{21}+4y_{24}^2|]^2\\
		&~~~+6|(x_{21}+4x_{24}^2)-(y_{21}+4y_{24}^2)|^2
		\\
		&\leq 12|x_{21}-y_{21}|^2+192|x_{24}-y_{24}|^2|x_{24}+y_{24}|^2.
			\end{split}
	\end{equation*}
Thus,		
			\begin{equation*}
			\begin{split}
		NU_2\leq& 12|x_{21}-y_{21}|^2+192|x_{24}-y_{24}|^2|x_{24}+y_{24}|^2\\
		&+3|x_{21}-y_{21}|^2+3|x_{24}-y_{24}|^2|x_{24}+y_{24}|^2+24|x_{24}-y_{24}|^2|x_{24}+y_{24}|^2\\
		&+|x_{21}-y_{21}|^2+|x_{24}-y_{24}|^2|x_{24}^4+x_{24}^3y_{24}+x_{24}^2y_{24}^2+x_{24}y_{24}^3+y_{24}^4|^2\\
		&+8|x_{24}-y_{24}|^2|x_{24}+y_{24}||x_{24}^4+x_{24}^3y_{24}+x_{24}^2y_{24}^2+x_{24}y_{24}^3+y_{24}^4|.
	\end{split}
\end{equation*}
Combining these results means that
	\begin{equation*}
		\begin{split}	
		NU_1+NU_2\leq	& 21\left|\left(\begin{array}{c}
				x_{11} \\
				x_{21}
			\end{array}\right)-\left(\begin{array}{c}
				y_{11} \\
				y_{21}
			\end{array}\right)\right|^2
			+\frac{3}{2}\left|\left(\begin{array}{c}
				x_{12} \\
				x_{22}
			\end{array}\right)-\left(\begin{array}{c}
				y_{12} \\
				y_{22}
			\end{array}\right)\right|^2\\
			&
			+590\left(1+\left|\left(\begin{array}{c}
				x_{14} \\
				x_{24}
			\end{array}\right)\right|^8+\left|\left(\begin{array}{c}
				y_{14} \\
				y_{24}
			\end{array}\right)\right|^8\right)\left|\left(\begin{array}{c}
				x_{14} \\
				x_{24}
			\end{array}\right)-\left(\begin{array}{c}
				y_{14} \\
				y_{24}
			\end{array}\right)\right|^2.
		\end{split}
	\end{equation*}
	Hence, Assumption \ref{Halpha} is satisfied with $$U_2\left(\left(\begin{array}{c}
		x_{14} \\
		x_{24}
	\end{array}\right),\left(\begin{array}{c}
	y_{14} \\
	y_{24}
\end{array}\right)\right)=1+\left|\left(\begin{array}{c}
x_{14} \\
x_{24}
\end{array}\right)\right|^8+\left|\left(\begin{array}{c}
y_{14} \\
y_{24}
\end{array}\right)\right|^8.$$ 
\end{exa}
In the numerical simulation, set $T=4$ and $N=1000$. The numerical solution with $\Delta=2^{-16}$ is regarded as the true solution, since the true solution cannot be expressed explicitly. The error is defined by
\begin{equation*}
		err=\left[\frac{1}{N}\sum_{i=1}^{N}|X_u^{i,N}(T)-X_{u_{l}}^{i,N}(T)|^2\right]^{\frac{1}{2}},
\end{equation*}	
where $u$ means the level of the time discretization, and $u_l\in \{u_1,u_2,u_3\}$ which match $\Delta=2^{-15}, 2^{-13},2^{-12},2^{-11}$.
From Figures \ref{tu1} and \ref{tu2}, we observe that the rate of convergence is about 0.5, which supports the findings.
\begin{figure}[htbp] 
	\centering
	\includegraphics[height=6.5cm,width=8cm]{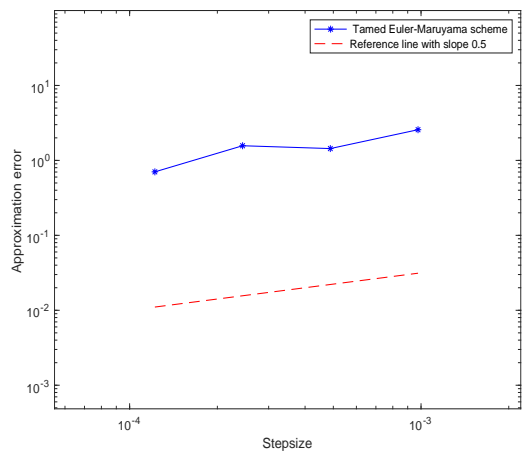}
	\caption{\label{tu2} The convergence rate of tamed EM scheme for (\ref{exam2})}
\end{figure}

	\section*{Acknowledgements}
This work is supported by the National Natural Science Foundation of China (Grant Nos. 12271368, 11871343 and 61876192) and Shanghai Rising-Star Program (Grant No. 22QA1406900).

\end{document}